\documentclass{birkmult}[12pt,a4paper]

\usepackage{mathrsfs}
\usepackage[centertags]{amsmath}
\usepackage{amsfonts}
\usepackage{amsthm}
\usepackage{color}
\usepackage{amssymb}

\usepackage{appendix}

\usepackage{lipsum}

\usepackage{amsmath}
\allowdisplaybreaks[4]

\usepackage{mathtools, nccmath}

\usepackage{fancyhdr}
\usepackage{amssymb,xcolor}

\newcommand\R{{\mathbb R}}
\renewcommand{\Re}{\mathrm{Re}}

\usepackage{color}

\advance\textwidth by 2.5cm
\advance\oddsidemargin by 0cm \advance\evensidemargin by 0cm
\theoremstyle{plain}

\newcommand{\DEQS}{\begin{eqnarray*}}
\newcommand{\EEQS}{\end{eqnarray*}}

\newcommand{\DEQSZ}{\begin{eqnarray}}
\newcommand{\EEQSZ}{\end{eqnarray}}
\newcommand{\RR}{{\mathbb{R}}}
\newcommand{\EE}{{\mathbb{E}}}

\newcommand{\CC}{\mathbb{C}}

\def\d{\text{\rm{d}}}
\def\i{\text{\rm{i}}}

\newtheorem{theorem}{Theorem}[section]

\newtheorem{proposition}[theorem]{Proposition}
\newtheorem{assumption}[theorem]{Assumption}
\newtheorem{definition}[theorem]{Definition}
\newtheorem{lemma}[theorem]{Lemma}

\newtheorem{remark}{Remark}[section]
\newtheorem{acknowledgements}{Acknowledgements}

\newtheorem{corollary}{Corollary}[section]
\input colordvi

\newcommand\del[1]{}

\newcommand\E{{\mathbb E}}

\begin{document}

\baselineskip 15.35pt
\numberwithin{equation}{section}
\title[Stochastic nonlinear Schr\"{o}dinger equations]
{{The stochastic Strichartz estimates and stochastic nonlinear
 Schr\"{o}dinger equations driven by L\'evy noise$^\dagger$}
 \footnote{$\dagger$ This work is supported in part by NSFC (12071433, 11822106, 11831014, 12090011, 11501509) and PAPD of Jiangsu Higher Education Institutions.}}
\author[Z. Brze\'{z}niak, W. Liu and J. Zhu]{}

\begin{center}
\begin{minipage}{145mm}

	\date{today}

\end{minipage}
\end{center}

\maketitle	

\newcommand\blfootnote[1]{%
\begingroup
\renewcommand\thefootnote{}\footnote{#1}%
\addtocounter{footnote}{-1}%
\endgroup
}

\centerline{{\scshape Zdzis{\l}aw Brze\'{z}niak$^a$, Wei Liu$^{b,c}$ and  Jiahui Zhu$^{d,*}$}\blfootnote{* Corresponding author: jiahuizhu@zjut.edu.cn}}
\medskip
 {\footnotesize
\centerline{ $a.$   Department of Mathematics, University of York,  York YO10 5DD, UK}
 \centerline{ $b.$ School of Mathematics and Statistics, Jiangsu Normal University, Xuzhou 221116, China}
  \centerline{$c.$ RIMS, Jiangsu Normal University, Xuzhou 221116, China}
 \centerline{  $d.$  School of Science, Zhejiang University of Technology, Hangzhou 310019, China }
 }

 \begin{abstract}
 We establish a new version of the stochastic Strichartz estimate for the stochastic convolution driven by jump noise which we apply
to the stochastic nonlinear Schr\"{o}dinger equation with nonlinear multiplicative jump noise in the Marcus canonical form.  With the help of the deterministic Strichartz estimates, we prove the existence and uniqueness of a global solution to stochastic nonlinear Schr\"{o}dinger equation in $L^2(\RR^n)$ with either focusing or defocusing nonlinearity in the full subcritical range of exponents as in the deterministic case.

\noindent \textbf{Keywords:}  Nonlinear Schr\"{o}dinger equation, Stochastic Strichartz estimate, L\'evy noise, Marcus canonical form.\\

\noindent \textbf{AMS Subject Classification.}  60H15; 60J75; 35B65; 46B09
 \end{abstract}


\section{Introduction}

In this paper, we study the following stochastic nonlinear Schr\"{o}dinger equation in the Marcus form in $L^{2}(\RR^{n})$
\begin{align}
\begin{split}\label{SNLSE-Marcus}
    \d u(t)&={\rm{i}}[\Delta u(t)-\lambda |u(t)|^{2\sigma}u(t)]\d t-\i\sum_{j=1}^mg_j(u(t-))\diamond \d L_j(t), \;\; t>0, \\
    u(0)&=u_{0},
    \end{split}
\end{align}
where $\lambda\in\mathbb{R}$, $ \sigma>0$ and $g_j:\mathbb{C}\rightarrow \mathbb{C}$ are measurable functions satisfying some additional conditions  and  $L(t)=(L_1(t),\cdots,L_m(t))$ is an $\mathbb{R}^m$-valued pure jump L\'{e}vy process.

The nonlinear Schr\"odinger equation (NLS) is a fundamental model for describing wave propagation which appears in various fields such as nonlinear optics, nonlinear water propagation, quantum physics, Bose-Einstein condensate, plasma physics and molecular biology. The global existence results of deterministic NLS are essentially obtained by a fixed point argument in a suitable mixed space along with Strichartz estimates and conservation laws. In a wide range of physical and engineering models it should be more appropriate to incorporate some types of random perturbations  which may be caused by the influence of thermal fluctuations or inhomogeneous media etc. In \cite{Bou+Deb1999} de Bouard and Debussche studied the existence and uniqueness of global $L^{2}(\mathbb{R}^{n})$-valued solutions to the stochastic NLS with linear multiplicative  Stratonovich noise. In the subsequent paper \cite{Bou+Deb2003}, they proved the global existence and uniqueness of solutions in the case of Stratonovich noise with paths in $H^{1,2}(\mathbb{R}^{n})$. Brze\'zniak and Millet in \cite{Brz+Mil2014} established a general version of Strichartz estimates for stochastic convolution and proved the global existence and uniqueness to the stochastic NLS with nonlinear Stratonovich noise in $H^{1,2}(M)$ for a two-dimensional compact Riemannian manifold $M$ which generalized the results of \cite{Burq+Gerard2004,Bou+Deb1999,Bou+Deb2003}. Based on the stochastic Strichartz estimate established in \cite{Brz+Mil2014}, Horhung in \cite{Hornung2018} studied the existence and uniqueness of solutions to the stochastic NLS with nonlinear Stratonovich noise in $L^{2}(\RR^{n})$ for subcritical and critical nonlinearities. By using a modified Faedo-Galerkin method,  Brze\'{z}niak et al. \cite{Brz+Hor+Weis2019} constructed a martingale solution for a stochastic NLS with multiplicative Wiener noise in an abstract framework and showed the pathwise uniqueness for the case of 2D compact manifolds with bounded geometry by means of the Strichartz estimates. The work \cite{Brz+Hor+Weis2019} was extended by  Horhung in \cite{Hornung2019} to  the non-compact manifolds case.
Barbu et al. \cite{Barbu+Rockner+Zhang2014,Barbu+Rockner+Zhang2017} proved the global well-posedness result of stochastic NLS with linear multiplicative Wiener noise via the rescaling approach and an application of the Strichartz estimates. See also \cite{Barbu+Rockner+Zhang2018,Herr+Rockner+Zhang2019,HW} and the references therein for some further studies on this topic.

Concerning the study of stochastic NLS perturbed by Gaussian noise,  the Stratonovich form has its merits as the conservation law is still preserved by the solution of the equation. Moreover, the Stratonovich form  has another two important features. The first one is that the Stratonovich form  obeys the classical rules of differentiation as in ordinary calculus. The second feature is that it is consistent with Wong-Zakai type approximation which is important from a modeling point of view. However, these remarkable properties of Stratonovich form are violated if  the driving process includes jumps because the higher order integrals do not vanish in the jump  noise case.  Marcus in \cite{Marcus1978,Marcus1981} introduced a new type of integral which pertains the same preferable properties as that of Stratonovich integral in the continuous case. We are thus motivated to study stochastic NLS of the Marcus canonical form \eqref{SNLSE-Marcus}. One most delightful property of the Marcus canonical form in \eqref{SNLSE-Marcus} is that it allows the equation to preserve the $L^{2}(\RR^{n})$-norm of the solution. This ensures the non-blow-up of the solution in finite time.

We study the global existence for two important types of stochastic NLS equations  for  the nonlinear term $\lambda |u(t)|^{2\sigma}u(t)$,  $\sigma\in(0,\frac2n)$,  which is either defocusing if  $\lambda>0$, or  focusing if  $\lambda<0$. In a forthcoming paper \cite{Brz+Zhu_2020} we will investigate related questions for a larger class of defocusing type nonlinearities.
 We prove a general version of Strichartz estimate for stochastic convolution driven by compensated Poisson random measures (see Proposition \ref{prop-str-ineq-11}). This Strichartz estimate for jump noise is new and novel. Compared to papers by de Bouard and Debussche \cite{Bou+Deb1999,Bou+Deb2003}, our Strichartz estimate holds for arbitrary admissible pair $(p,r)$ and hence avoids the additional restriction $0<\sigma<\frac1{n+1}$ for $n\geq 3$ on $\sigma$.  On the basis of this stochastic Strichartz estimate, our global existence and uniqueness result is established in the full range of subcritical case which is consistent with the deterministic NLS in $L^2$, see e.g. \cite{Linares+Ponce}.
 This general version of the Strichartz estimate also allows us to obtain the existence and uniqueness results to the stochastic NLS with nonlinear noise which differs from the linear noise results from \cite{Brz+Hor+Manna} and \cite{Bou+Hau2019}. As the nonlinear term is not Lipschitz, we shall truncate the nonlinear term and combine the Strichartz estimates for the approximate solution to construct a solution of the truncated problem by using a fixed point argument in the space $ D(0,T;L^2(\RR^n))\cap L^p(0,T;L^r(\RR^n))$. A novel aspect of our equation \eqref{SNLSE-Marcus} in the Marcus canonical form is the conservation of $L^{2}(\RR^{n})$-norm and the global existence is a consequence of the $L^{2}(\RR^{n})$-norm-preserving condition. We emphasize that our stochastic Strichartz estimates are useful not only for equations in the Marcus canonical form. Indeed, they can be used to study stochastic NLS in the It\^o form. The advantage of the Marcus canonical form is twofold. Firstly, the solutions to such equations preserve a.s.  the $L^2$-norm. Secondly, they are believed to be robust with respect to the approximation of the noise.
In the case of the gaussian noise, related research is contained in recent papers \cite{Brz+Man+Mukh2019},  \cite{Fan+Xu_2019},  \cite{Gu+Tsai_2019} and \cite{Man+Mukh+Panda2019}.

In spite of quite a number of contributions on stochastic NLS with the Gaussian noise, the theory is much less well-developed in the case where the driving noise has jumps. Recently de Bouard and Hausenblas in \cite{Bou+Hau2019} studied the existence of martingale solutions of the stochastic NLS driven by a L\'evy-type noise but they considered the case of focusing nonlinearity and linear noise in $H^{1,2}(\RR^{n})$.  Brze\'zniak et al. in \cite{Brz+Hor+Manna} constructed a martingale solution of the stochastic NLS with a multiplicative jump noise by using a variant of the Faedo-Galerkin method. Comparing our work with \cite{Brz+Hor+Manna},  we would like to point out three main differences. Firstly, the current paper establishes the stochastic Strichartz estimates for the stochastic NLS driven by Poisson random measures and use them to prove  the existence and uniqueness of strong solutions to stochastic nonlinear Schr{\"o}dinger equation,  while the other paper proves the existence (but  not the uniqueness) of  weak martingale solutions to stochastic NLS using the compactness method and the generalization of the Skorokhod-Jakubowski theorem from \cite{Brz+Haus+Raza_2018} and \cite{Motyl_2012}. Secondly, the current paper deals with  solutions  with  the initial data from  the $L^2$ space, while the other paper deals with solutions with initial data belongs to the energy space $H^1$.  Thirdly, the current paper is setup in the whole Euclidean space while the other paper is on a compact Riemannian manifold. The common feature of both papers is the use of stochastic equations with respect to L\'evy noise in the Marcus canonical form, which is in comparison with \cite{Bou+Hau2019}. An important consequence of this is that the solutions constructed in both papers have their mass, i.e. the $L^2$-norm, a.s. preserved. Let us point out that Brze\'{z}niak and Manna \cite{Brz+Man2019} recently proved the existence of weak martingale solutions for stochastic Landau-Lifshitz-Gilbert equations with pure jump noise also in the Marcus canonical form.

We believe that the study of the stochastic Strichartz estimate in this paper may generate new insights about SPDEs with jump noise in the future. For instance, recently the authors \cite[Proposition 3.4]{Br+Liu+Zhu-SPL-2019} proved the boundedness of stochastic convolution of Stokes operator in $L^{4}(0,T;L^{4}(D))$ in a very similar spirit and established the existence and uniqueness of solutions of 2D stochastic Navier-Stokes equation in certain Sobolev spaces of negative order. See also \cite{Br+Liu+Zhu2019} for similar results of stochastic quasi-geostrophic equations.  We also expect that our investigation of Marcus canonical stochastic NLS equation may shed some lights on the analysis  of stochastic partial differential equations perturbed by jump noise.

Let us formulate our main result of this paper.
          \begin{theorem}\label{main-theo}Let $p\geq 2$, $0<\sigma<\frac{2}{n}$, $r=2\sigma+2$  such that  $(p,r)$ is an admissible pair,  i.e. the conditions (\ref{eqn-admissible pair-01}-\ref{eqn-admissible pair-02}) below are satisfied, and $u_0\in L^p(\Omega,\mathcal{F}_{0};L^2(\RR^n))$.
          Under Assumption \ref{assu-1}, there exists a unique global mild solution  $u=(u(t))$, $t\in[0,T]$ of \eqref{SNLSE-Marcus} such that $$u\in L^{p}\big(\Omega; L^{\infty}(0,T;L^2(\RR^n))\cap L^p(0,T;L^r(\RR^n))\big),$$
          and $u(\cdot,\omega)\in D(0,T;L^2(\RR^n))$ for $\mathbb{P}$-a.a. $\omega\in\Omega$. Moreover, we have for all $t\in[0,T]$,
           $$\|u(t)\|_{L^2(\RR^n)}=\|u_0\|_{L^2(\RR^n)},\quad\mathbb{P}\text{-a.s.} $$
          \end{theorem}

     The rest of the paper is organized as follows. In Section 2 we introduce some notations and assumptions and investigate the stochastic Strichartz estimate. Section 3 is devoted to studying the truncated equation and proving the local existence and uniqueness of \eqref{NSS-integral-form}. In the last section, we first prove that the $L^{2}$-norm of the solution in \eqref{SNLSE-Marcus} is conserved and get a uniform estimate of the solutions of the truncated equations in  $L^{p}(\Omega;L^p(0,T;L^r(\RR^n)))$. Finally, we establish the existence and uniqueness of global solutions.

\section{Notations and Strichartz estimates}
Let  $(\Omega,\mathcal{F},\mathbb{F},\mathbb{P})$, where $\mathbb{F} = (\mathcal{F}_t )_{t\geq 0}$, be a filtered probability space satisfying the usual hypothesis. Let $L(t)=(L_1(t),\cdots,L_m(t))$, $t\geq0$ be an $\mathbb{R}^m$-valued pure jump L\'{e}vy process with L\'{e}vy measure $\nu$, i.e. $L(t)=\int_0^t\int_{B}z\tilde{N}(\d s,\d z)$, where $B=\{z\in\mathbb{R}^m:0<|z|\leq 1\}$ and $N$ is a time homogeneous Poisson random measure on $\RR^{+}\times(\mathbb{R}^{m}-\{0\})$ with $\sigma$-finite intensity measure $\nu$ satisfying $\int_{B}|z|^{2}\nu(\d z)<\infty$.

Based on the definition of the Marcus canonical integral from \cite{Marcus1978,Marcus1981}, equation \eqref{SNLSE-Marcus} with the notation $\diamond$ is defined by
\begin{align}
\begin{split}\label{NSS-integral-form}
     \d u(t)=&\i[\Delta u(t)-\lambda |u(t)|^{2\sigma}u(t)]\d t+\int_B\Big[\Phi(z,u(t-))-u(t-)\Big]\tilde{N}(\d t,\d z)\\
     &+\int_B\Big[\Phi(z,u(t))-u(t)+\i\sum_{j=1}^mz_jg_j(u(t))\Big]\nu(\d z)\d t,\quad t>0,
     \end{split}
\end{align}
with $\Phi(z,x)$ being the value  at time $t=1$ of the solution of the following equation
\begin{align}\label{phi-equation}
\frac{\partial \Phi}{\partial t}(t,z,x)=-\i\sum_{j=1}^mz_jg_j(\Phi(t,z,x)),\quad \Phi(0,z,x)=x.
\end{align}

We make the following assumptions.
\begin{assumption}\label{assu-1}
\begin{enumerate}
\item[(1)]  Let $T\in(0,\infty)$, $n\in\mathbb{N}$ and $(S_t)_{t\in\RR}$ denote the group of isometries on $L^{2}(\RR^{n})$ generated by the operator $\i\Delta$.
\item[(2)] For each $1\leq j\leq m$, there exists a function $\tilde{g}_{j}:[0,\infty)\rightarrow\RR$ of class $C^{1}$ such that $g_{j}$ is given by $g_j(y)=\tilde{g}_j(|y|^2)y$, $y\in\mathbb{C}$. We also assume there exist constants $L_{1},L_{2}>0$ such that  for all $x,y\in\mathbb{C}$
    \begin{align}
    \max_{1\leq j\le m}|g_j(x)-g_j(y)|&\leq L_{1}|x-y|,\label{Lip_of_g}\\
    \max_{1\leq j,k\leq m}|g_j'(x)g_k(x)-g_j'(y)g_k(y)|&\leq L_{2}|x-y|.\label{Lip_of_m_ji}
    \end{align}
\end{enumerate}
\end{assumption}
 \begin{remark}
 \begin{enumerate}
 \item[(1)]We identify below $\CC$ with $\RR^2$ and denote $\langle \cdot, \cdot\rangle$ (resp. $|\cdot|$)  the scalar product (resp. the Euclidean norm) in $\CC \cong \RR^2$.
      Let us introduce the linear operator $\mathcal{I}:\RR^2\ni(y_{1},y_{2})\mapsto (-y_{2},y_{1})\in\RR^2$
      and identify the operator of multiplication in $\mathbb{C}$ by the imaginary unit $\i$ with the operator $\mathcal{I}$ in $\mathbb{R}^2$. Given a function $g_i$ as in Assumption \ref{assu-1} we define $\phi_{j}(y):=(\i g_j)(y)=\tilde{g}_{j}(|y|^2)\mathcal{I}y \in \mathbb{R}^2$, for $y\in \mathbb{R}^2$. Then we have
 \begin{align*}
     [\phi'_j(y)](x)=2\tilde{g}'_j(|y|^2)\langle y,x\rangle \mathcal{I}y+\tilde{g}_j(|y|^2)\mathcal{I}x, \quad\text{for }x,y\in\mathbb{R}^{2}.
 \end{align*}
 From this it follows that
 \begin{align*}
        [(\i g_{j})'(y)](-\i g_k(y))&=2\tilde{g}'_j(|y|^2)\langle y,-\mathcal{I}g_k(y)\rangle \mathcal{I}y-\tilde{g}_j(|y|^2)\mathcal{I}^2g_k(y)\\
        &=-2\tilde{g}'_j(|y|^2)\tilde{g}_k(|y|^2)\langle y,\mathcal{I}y\rangle \mathcal{I}y+\tilde{g}_j(|y|^2)\tilde{g}_k(|y|^2)y\\
        &=\tilde{g}_j(|y|^2)\tilde{g}_k(|y|^2)y =m_{j,k}(y),
 \end{align*}
 where $m_{jk}(y)=\tilde{g}_j(|y|^2)\tilde{g}_k(|y|^2)y$. Thus the above condition \eqref{Lip_of_m_ji} is equivalent to
 \begin{align*}
&\max_{1\leq j,k\leq m}|m_{jk}(x)-m_{jk}(y)|\leq L_{2}|x-y|,\quad x,y\in \CC.
 \end{align*}
 For Assumption \ref{assu-1}$(2)$ to hold, it is sufficient that each $\tilde{g}_{j}\in C_{b}^{2}([0,\infty);\mathbb{R})$ and $\sup_{\theta>0}(1+\theta)|\tilde{g}_{j}'(\theta)|<\infty$, for $j=1,\cdots,m$.
\item[(2)] Under Assumption \ref{assu-1}, there exists a measurable mapping $\Phi:\mathbb{R}^+\times \mathbb{R}^m\times \mathbb{C}\rightarrow \mathbb{C}$ such that, for each $z\in\mathbb{R}^m$, $x\in\CC$, the function $t\mapsto \Phi(t,z,x)$ is continuously differentiable and solves \eqref{phi-equation}. In virtue of Lemma \ref{lem-stoch-lip}, the stochastic integral and Lebesgue integral in \eqref{NSS-integral-form} are well defined.
\end{enumerate}
    \end{remark}

We now state the following famous deterministic Strichartz estimates, we refer the reader to e.g. \cite[Theorem  2.3.3]{Caz2003}  for the proof. Let us first recall the definition of an admissible pair. We say a pair $(p,r)$ is admissible if
  \begin{align}\label{eqn-admissible pair-01}
  \frac2p=n\Big(\frac12-\frac1r\Big)
  \end{align}
  and
    \begin{align}\label{eqn-admissible pair-02}
  \begin{cases}
  2\leq r \leq \infty,\quad  &\mbox{if}  \quad n=1,\\
   2\leq r <\infty, \quad &\mbox{if} \quad n=2,\\
     2\leq r \leq \frac{2N}{N-2},\; & \mbox{if}\quad  n\geq 3.
  \end{cases}
  \end{align}
  Notice that $(\infty,2)$ and $(2,\frac{2n}{n-2})$, $n\geq3$ are always admissible.

\begin{proposition}\label{prop-strichartz-1}
Let $(p,r)$ and $(\gamma,\rho)$ be two admissible pairs and let $\gamma',\rho'$  be conjugates of $\gamma$ and $\rho$. Then
\begin{enumerate}
\item[(1)] for every $\phi\in L^2(\RR^n)$, the function $t\mapsto S_t\phi$ belongs to $L^p(\RR;L^r(\RR^n))\cap L^{\infty}(\RR;L^2(\RR^n))$ and there exists a constant $C$  such that
\begin{align}\label{eq-stri-1}
\|S_\cdot\phi\|_{L^p(\RR;L^r(\RR^n))}\leq C\|\phi\|_{L^2(\RR^n)}.
\end{align}
\item[(2)] Let $I$ be an interval of $\RR$ and $J=\bar{I}$ with $0\in J$. Then for every $f\in L^{\gamma'}(I;L^{\rho'}(\RR^n))$, the function $t\mapsto \Phi_f(t)=\int_0^tS_{t-s}f(s)\d s$ belongs to $L^p(I;L^r(\RR^n))\cap L^{\infty}(J;L^2(\RR^n))$ and there exists a constant $C$ independent of $I$  such that
\begin{align}
&\|\Phi_f\|_{L^{\infty}(J;L^2(\RR^n))}\leq C\|f\|_{L^{\gamma'}(I;L^{\rho'}(\RR^n))} ;\label{det-stri-ine-L2}\\
&\|\Phi_f\|_{L^p(I;L^r(\RR^n))}\leq C\|f\|_{L^{\gamma'}(I;L^{\rho'}(\RR^n))} .\label{det-stri-ine}
\end{align}
\end{enumerate}

\end{proposition}


\begin{remark}\label{remark-stri-eq-1}
     Note that inequality \eqref{eq-stri-1}
      is  consistent with Assumption 3.1 in \cite{Brz+Mil2014}. Take $v\in L^{\infty}(0,T;L^2(\RR^n))$, by  \eqref{eq-stri-1} we have
     \begin{align*}
         \|S_{\cdot-s}1_{[s,T]}(\cdot)v_s\|_{L^p(0,T;L^r(\RR^n))}&=\Big(\int_s^T\|S_{t}S_{-s}v_s\|^p_{L^r(\RR^n)}\d t\Big)^{\frac1p}\\
      &   \leq \Big(\int_0^T\|S_{t}S_{-s}v_s\|^p_{L^r(\RR^n)}\d t\Big)^{\frac1p}
         \leq C\|v_s\|_{L^2(\RR^n)}.
     \end{align*}
     It follows that
     \begin{align*}
            \Big\| \int_0^t S_{t-s}v_s\,\d s\Big\|_{L^p(0,T;L^r(\RR^n))}&=\Big\| \int_0^T1_{[s,T]}(t)S_{t-s}v_s\,\d s\Big\|_{L^p(0,T;L^r(\RR^n))}\\
            &\leq \int_0^T\|1_{[s,T]}(t)S_{t-s}v_s \|_{L^p (0,T;L^r(\RR^n))}\d s \\
            &\leq  C  \int_0^T \|v_s\|_{L^2(\RR^n)}\d s\\
            &\leq T C\|v\|_{L^{\infty}(0,T;L^2(\RR^n))}.
     \end{align*}
    This inequality will play a key role later.

\end{remark}
    Throughout the paper, the symbol $C$ will denote a positive generic constant whose value may change from place to place. If a constant depends on some variable parameters, we will put them in subscripts.

Let $p,r\in[2,\infty)$ with $\frac{2}{p}=n(\frac12-\frac1r)$, that is $(p,r)$ is an admissible pair.
For $0\leq t_1\leq t_2$, let us denote $D(t_1,t_2;L^2(\mathbb{R}^n))$ the space of all right continuous functions with left-hand limits from $[t_{1},t_{2}]$ to $L^2(\mathbb{R}^n)$ and
\begin{align}
Y_{[t_1,t_2]}:=L^{\infty}(t_1,t_2;L^2(\mathbb{R}^n))\cap L^p(t_1,t_2;L^r(\mathbb{R}^n)).
\end{align}
Then $Y_{[t_{1},t_{2}]}$ is a Banach space with norm defined by
  \begin{align}
   \|u\|_{Y_{[t_{1},t_{2}]}}:=\sup_{s\in[t_{1},t_{2}]}\|u(s)\|_{L^2(\mathbb{R}^n)}+\Big(\int_{t_{1}}^{t_{2}}\|u(s)\|_{L^r(\mathbb{R}^n)}^p\d s\Big)^{\frac1p}.
  \end{align}
For the simplicity of notation, we write $Y_t$ instead of $Y_{[0,t]}$.   Notice that $Y_t$, $t\geq 0$ is a non-decreasing family of Banach spaces. That is if $0<s<t$ and $u\in Y_t$, then $u|_{[0,s]}\in Y_s$ and $\|u|_{[0,s]}\|_{Y_s}\leq \|u\|_{Y_{t}}$.
  Let $\tau>0$ be a stopping time. We call $\tau$ an accessible stopping time if there exists an increasing sequence $(\tau_n)_{n\in\mathbb{N}}$ of stopping times such that $\tau_{n}<\tau$ and $\tau_{n} \nearrow \tau$ $\mathbb{P}$-a.s. as $n\rightarrow\infty$ and we call $(\tau_n)_{n\in\mathbb{N}}$ an approximating sequence for $\tau$.

   Let $M^p_{\mathbb{F}}(Y_{\tau}):=L^{p}(\Omega; L^{\infty}(0,\tau;L^2(\RR^n))\cap L^p(0,\tau;L^r(\RR^n)))$ denote the space of all $L^{2}(\RR^{n})\cap L^{r}(\RR^{n})$-valued $\mathbb{F}$-progressively measurable processes $u:[0,T]\times \Omega \rightarrow L^2(\RR^n)\cap L^{r}(\RR^{n})$  satisfying
  \begin{align*}
  \|u\|^p_{M^p_{\mathbb{F}}(Y_{\tau})}:=\EE\|u\|^p_{Y_\tau}=\EE\Big(\sup_{s\in[0,\tau]}\|u(s)\|^{p}_{L^2(\mathbb{R}^n)}+\int_0^{\tau}\|u(s)\|_{L^r(\mathbb{R}^n)}^p\d s\Big)<\infty.
  \end{align*}

Now we introduce the definitions of local solutions and maximal local solutions, see e.g. \cite{Brz+Zhu2016} for more details.
  \begin{definition}
        A local mild solution to equation \eqref{NSS-integral-form} is an $\mathbb{F}$-progressively measurable process $u(t)$, $t\in[0,\tau)$, where $\tau$ is an accessible stopping time with an approximating sequence $(\tau_n)_{n\in\mathbb{N}}$ of stopping times such that for every $n\in\mathbb{N}$,
       \begin{enumerate}
        \item[(i)] $(u_t)_{t\in[0,\tau_{n}]}\in D(0,\tau_{n};L^{2}(\mathbb{R}^{n}))$, $\mathbb{P}$-a.s.;
\item[(ii)] $(u(t))_{t\in[0,\tau_{n}]}$ belongs to $M^p_{\mathbb{F}}(Y_{\tau_{n}})$;
\item[(iii)] for every $t\in [0,T]$, the following equality holds
        \begin{align*}
          u(t\wedge \tau_{n})=&S_{t\wedge\tau_{n}}u_0-\i\int_0^{t\wedge\tau_{n}}S_{t\wedge\tau_{n}-s}(\lambda |u(s)|^{2\sigma}u(s))\,\d s+I_{\tau_n}(\Phi(z,u)-u)(t\wedge\tau_{n})\nonumber\\
     &+\int_0^{t\wedge\tau_{n}}\int_BS_{t\wedge\tau_{n}-s}\Big[\Phi(z,u(s)))-u(s)+\i\sum_{j=1}^mz_jg_j(u(s))\Big]\nu(\d z)\d s, \;\;\mathbb{P}\text{-a.s. }
        \end{align*}
          where $I_{\tau_n}(\Phi(z,u)-u)$ is a process defined by
          \begin{align*}
         I_{\tau_n}(\Phi(z,u)-u)(t)= \int_0^t\int_B1_{[0,\tau_{n}]}S_{t-s}\Big[\Phi(z,u(s\wedge\tau_{n}-))-u(s\wedge\tau_{n}-)\Big]\tilde{N}(\d s,\d z).
          \end{align*}
        \end{enumerate}
        A local mild solution $(u(t))_{t\in[0,\tau)}$ is called unique, if for any other local mild solution $(v(t))_{t\in[0,\sigma) }$ of \eqref{NSS-integral-form}, we have
        \begin{align*}
            \mathbb{P}(u(t)=v(t),\;\;\forall t\in[0,\tau\wedge\sigma))=1.
        \end{align*}
         A local mild solution $u=(u(t))_{0\leq t<\tau}$ is called a maximal local mild solution if for any other local mild
   solution $(v(t))_{t\in[0,\sigma) }$
   satisfying $\sigma\geq \tau$ a.s. and $v|_{[0,\sigma)}$ is equivalent to $
   u$, one has  $\sigma=\tau$ a.s.. \\
           A local mild solution $(u(t))_{t\in[0,\tau)}$ is a global mild solution if $\tau=T$, $\mathbb{P}$-a.s. and $u\in M^p_{\mathbb{F}}(Y_T)$.
  \end{definition}

Now let us explore the stochastic Strichartz estimates. First we recall the definition of martingale type $2$ Banach space.
We say a real separable Banach space $(E,\|\cdot\|_E)$ is of martingale type
$2$ if there is a constant $K(E)>0$ such that
for all $E$-valued discrete martingales $\{M_n\}_{n=0}^N$ the
following inequality holds
\begin{align*}
    \sup_{n}\E\|M_n\|_{E}^2\leq K(E)\sum_{n=0}^N\E\|M_n-M_{n-1}\|_{E}^2,
\end{align*}
where we set $M_{-1}=0$ as usual. Note that all $L^p$ spaces, $p\geq
2$ are of martingale type $2$.

Let $E$ be a separable Banach space of martingale type $2$ and let $\xi:[0,T]\times\Omega \times Z\rightarrow E$ be an $E$-valued $\mathbb{F}$-predictable process in $L^{2}([0,T]\times \Omega\times Z)$.  For detailed discussion of stochastic integral with respect to Poisson random measure in martingale type $2$ Banach space,  we refer to \cite{Zhu+Brz+Hau2017,Br+Liu+Zhu2019}.  The Burkholder inequality holds in this framework, i.e.  there exists a generic constant $C_{q}$ depending only on $q$ and the constant $K(E)$ from the martingale type 2 condition such that (see e.g. \cite[Corollary 2.4]{Br+Liu+Zhu2019})
 \begin{align}
		    &\EE\sup_{t\in[0,T]}\Big\|\int_0^t\int_Z\xi(s,z)\tilde{N}(\d s,\d z)\Big\|_E^{q}\leq C_{q}\,\EE\Big(\int_0^T\int_Z \|\xi(s,z)\|^q_E\;\nu(\d z)\d s\Big)\nonumber\\
		 &\hspace{6cm}+C_{q}\,\EE \Big(\int_0^T\int_Z\|\xi(s,z)\|^{2}_E\;\nu(\d z)\d s\Big)^{\frac{q}{2}}\text{  for all }2\leq q<\infty.\label{intro-eq-main-1}
	\end{align}
	\begin{lemma}Suppose that $F$ is a martingale type $2$ Banach space. Let $p\in[2,\infty)$ and $\xi:[0,T]\times\Omega \times Z\rightarrow  L^p(0,T;F)$ be an $\mathbb{F}$-predictable process.  For all $q\geq 2$, we have
	 \begin{align}
		    \EE\Big\|\int_0^T\int_Z\xi(s,z)\tilde{N}(\d s,\d z)\Big\|_{L^p(0,T;F)}^{q}\leq& C_{q}\,\EE\Big(\int_0^T\int_Z \|\xi(s,z)\|^q_{L^p(0,T;F)}\;\nu(\d z)\d s\Big)\nonumber\\
		 &+C_{q}\,\EE \Big(\int_0^T\int_Z\|\xi(s,z)\|^{2}_{L^p(0,T;F)}\;\nu(\d z)\d s\Big)^{\frac{q}{2}}.\label{Burk-Lp}
	\end{align}
	\end{lemma}
	\begin{proof}
	The proof of  this lemma is an immediate consequence of the facts that $L^p(\RR^+;F)$ is a martingale type 2 Banach space, and since $L^p(0,T;F)$ is isometrically identified with a closed subspace of $L^p(\RR^+;F)$, it is still a martingale type 2 Banach space.
	\end{proof}

\begin{proposition}\label{prop-str-ineq-2} Let $\xi:[0,T]\times\Omega \times Z\rightarrow L^2(\RR^n;\CC)$ be an $\mathbb{F}$-predictable process.  For all $q\geq 2$, we have
\begin{align}\label{str-L2-ineq}
\EE\sup_{t\in[0,T]}\Big\|\int_0^t\int_Z S_{t-s}\xi(s,z)\tilde{N}(\d s,\d z)\Big\|^q_{L^2(\RR^n)}\leq &C_{q}\,\EE\int_0^T\int_Z\|\xi(s,z)\|^q_{L^2(\mathbb{R}^n)}\nu(\d z)\d s\nonumber\\
&+C_{q}\,\EE\Big(\int_0^T\int_Z\|\xi(s,z)\|^2_{L^2(\mathbb{R}^n)}\nu(\d z)\d s\Big)^{\frac{q}{2}}.
\end{align}
\end{proposition}
\begin{proof}
Note that $(S_t)_{t\in\RR}$ is a unitary group in $L^2(\mathbb{R}^n)$, so $(S_t)_{t\in\RR}$ is a $C_0$-group of contractions. The required result \eqref{str-L2-ineq} follows from a straightforward application of the maximal inequality  in  \cite[Theorem 3.1]{Br+Liu+Zhu2019}.
\end{proof}

Now we shall establish a stochastic Strichartz inequality for
the stochastic convolutions driven by the compensated Poisson random measures which will play a major role in the study of stochastic NLS.
\begin{proposition}\label{prop-str-ineq-11}
Let $(p,r)$ be an admissible pair and  $p,r\in[2,\infty)$. Then for all $q\geq 2$ and all $\mathbb{F}$-predictable process $\xi:[0,T]\times\Omega \times Z\rightarrow L^{r}(\RR^{n};\mathbb{C})$ in $L^{q}\big(\Omega;L^{2}([0,T]\times Z;L^{2}(\mathbb{R}^{n}))\cap L^{q}([0,T]\times Z;L^{2}(\mathbb{R}^{n}))\big)$, we have
\begin{align}\label{str-ineq-stoch}
\EE\Big\|\int_0^\cdot\int_Z S_{\cdot-s}\xi(s,z)\tilde{N}(\d s,\d z)\Big\|^q_{L^p(0,T;L^r(\RR^n))}\leq &C_{q}\,\EE\Big(\int_0^T\int_{Z}\|\xi(s,z)\|^2_{L^2(\RR^n)}\nu(\d z)\d s\Big)^{\frac{q}{2}}\nonumber\\
&+C_{q}\,\EE\Big(\int_0^T\int_{Z}\|\xi(s,z)\|^q_{L^2(\RR^n)}\nu(\d z)\d s \Big).
\end{align}

\end{proposition}
\begin{proof}
For an $L^{r}(\mathbb{R}^{n})$-valued $\mathbb{F}$-predictable process $\xi$, we define an $L^{p}(0,T;L^{r}(\RR^{n}))$-valued function $\Upsilon$ as follows
\begin{align}
     \Upsilon_{s,z}:=\big\{[0,T]\ni t\mapsto 1_{[s,T]}(t)S_{t-s}\xi(s,z)\big\},\quad s\in[0,T],\,z\in Z.
\end{align}
Note that the mapping $\Gamma_{s}:L^{2}(\RR^{n})\ni x\mapsto 1_{[s,T]}(\cdot)S_{t-s}x\in L^{p}(0,T;L^{r}(\RR^{n}))$ is linear and continuous and $\Upsilon_{s,z}(\omega)=\Gamma_{s}\circ \xi(s,z,\omega)$, $(s,z,\omega)\in [0,T]\times Z\times \Omega$. So the process $\Upsilon_{s,z}$ is $\mathbb{F}$-predictable.
 Applying inequality \eqref{Burk-Lp} gives
\begin{align*}
&\EE\left\|\int_0^\cdot\int_Z S_{\cdot-s}\xi(s,z)\tilde{N}(\d s,\d z)\right\|^q_{L^p(0,T;L^r(\RR^n))}\\
&=\EE\Big(\int_0^T\Big\|\int_0^T\int_Z1_{[s,T]}(t) S_{t-s}\xi(s,z)\tilde{N}(\d s,\d z)\Big\|^p_{L^r(\RR^n)}\d t\Big)^{\frac{q}{p}}\\
&=\EE\Big(\int_0^T\Big\|\int_0^T\int_Z
\Upsilon_{s,z}(t)\tilde{N}(\d s,\d z)\Big\|^p_{L^r(\RR^n)}\d t\Big)^{\frac{q}{p}}\\
&= \EE\Big\|\int_0^T\int_Z
\Upsilon_{s,z}(\cdot)\tilde{N}(\d s,\d z)\Big\|^q_{L^p(0,T;L^r(\RR^n))}\\
&\leq  C_q  \EE\Big(\int_0^T\int_Z
\|\Upsilon_{s,z}(\cdot)\|^2_{L^p(0,T;L^r(\RR^n))}\nu(\d z)\d s\Big)^{\frac{q}{2}}\\
&\hspace{1cm}+C_q \EE\Big(\int_0^T\int_Z
\|\Upsilon_{s,z}(\cdot)\|^q_{L^p(0,T;L^r(\RR^n))}\nu(\d z)\d s\Big).
\end{align*}
By using the Strichartz inequality \eqref{eq-stri-1},
we have
\begin{align*}
\|\Upsilon_{s,z}(\cdot)\|_{L^p(0,T;L^r(\RR^n))}&=\| 1_{[s,T]}(\cdot) S_{\cdot-s}\xi(s,z)\|_{L^p(0,T;L^r(\RR^n))}\\
&\leq \|  S_{\cdot-s}\xi(s,z)\|_{L^p(0,T;L^r(\RR^n))}\\
&\leq C\|S_{-s}\xi(s,z)\|_{L^2(\RR^n)}\\
&\leq C\|\xi(s,z)\|_{L^2(\RR^n)}.
\end{align*}
Inserting back gives
\begin{align*}
   \EE\Big\|\int_0^\cdot\int_Z S_{\cdot-s}\xi(s,z)\tilde{N}(\d s,\d z)\Big\|^q_{L^p(0,T;L^r(\RR^n))}\leq  &C_{q}\EE\Big(\int_0^T\int_Z
\|\xi(s,z))\|^2_{L^2(\RR^n)}\nu(\d z)\d s\Big)^{\frac{q}{2}}\\
&+C_q \EE\Big(\int_0^T\int_Z
\|\xi(s,z)\|^q_{L^2(\RR^n)}\nu(\d z)\d s\Big).
\end{align*}

\end{proof}

\begin{remark} Although we only consider the stochastic NLS equations in this paper, similar estimates could be achieved for other stochastic PDE models. For example, the authors  in \cite{Br+Liu+Zhu2019,Br+Liu+Zhu-SPL-2019}  studied the existence and uniqueness of mild solutions for the 2D stochastic Navier-Stokes equation and the stochastic 2D quasi-geostrophic equation by separating the equations into a deterministic nonlinear PDE and a linear stochastic PDE respectively. With some technical modifications, the above arguments can be carried over to show the $L^{4}(0, T;L^{4}(\mathbb{R}^{2}))$-integrability of the stochastic convolutions of the linear stochastic PDEs of these two equations, which would guarantee the existence of solutions in $L^{4}(0, T;L^{4}(\mathbb{R}^{2}))$ to the stochastic Navier-Stokes equation and the stochastic 2D quasi-geostrophic equation.
\end{remark}
Put $T_{0}\geq 0$. Denote by $\mathbb{F}^{T_{0}}:=(\mathcal{F}_{t+T_{0}})_{t\geq0}$ the shifted filtration.  Define a new process by
\begin{align*}
N_{T_{0}}(t,A):=N(t+T_{0},A)-N(T_{0},A)
\end{align*}
for each $t\geq0$ and $A\in \mathcal{Z}$. It is easy to verify that $N_{T_{0}}(t,A)$ is a Poisson random measure with respect to $\mathbb{F}^{T_{0}}$ with the same intensity measure $\nu$ and
\begin{align}\label{changle-variable-SC}
\int_{0}^{t}\int_{Z}S_{t-s}g(T_{0}+s,z)\tilde{N}_{T_{0}}(\d s,\d z)=\int_{T_{0}}^{T_{0}+t}\int_{Z}S_{T_{0}+t-s}g(s,z)\tilde{N}(\d s,\d z).
\end{align}

\begin{corollary}\label{prop-str-ineq-1} Let $T_0\geq0$ and $T_{1}>0$. Assume that $(p,r)$ is an admissible pair and $p,r\in[2,\infty)$. Then for all $q\geq 2$ and all $\mathbb{F}^{T_{0}}$-predictable process $\xi:[0,T_{1}]\times\Omega \times Z\rightarrow L^2(\RR^n;\CC)$ in $L^{q}(\Omega;L^{2}([0,T_{1}]\times Z);L^{r}(\mathbb{R}^{n}) \cap L^{q}([0,T_{1}]\times Z);L^{2}(\mathbb{R}^{n}))$,
\begin{align}\label{str-ineq-stoch-2}
\EE\Big\|\int_0^\cdot\int_Z S_{\cdot-s}\xi(s,z)\tilde{N}_{T_{0}}(\d s,\d z)\Big\|^q_{L^p(0,T_{1};L^r(\RR^n))}\leq &C_{q}\Big[\EE\Big(\int_0^{T_1}\int_{Z}\|\xi(s,z)\|^2_{L^2(\RR^n)}\nu(\d z)\d s\Big)^{\frac{q}{2}}\nonumber\\
&+\EE\Big(\int_0^{T_1}\int_{Z}\|\xi(s,z)\|^q_{L^2(\RR^n)}\nu(\d z)\d s\Big)\Big].
\end{align}

\end{corollary}

\section{A truncated equation}

  In this section, we will construct a local solution of equation \eqref{NSS-integral-form}. Since the nonlinear term is not Lipschitz, we will use a similar truncation argument as in \cite{Brz+Mil2014,Bou+Deb1999,Hornung2018} and approximate the original equation by truncating the nonlinear term as follows.
  First we define a truncation function $\theta$. Let $\theta:\mathbb{R}_+\rightarrow[0,1]$ be a non-increasing $C_0^{\infty}$ function such that
  $1_{[0,1]}\leq \theta \leq 1_{[0,2]}$ and $\inf_{x\in\RR_+}\theta'(x)\geq-1$. For $R\geq 1$, set $\theta_R(\cdot)=\theta(\frac{\cdot}{R})$.
\begin{remark}\label{rem-theta}If $h:\RR_+\rightarrow\RR_+$ is a non-decreasing function, then for every $x,y\in\RR$,
\begin{align*}
\theta_R(x)h(x)\leq h(2R),\;\; |\theta_R(x)-\theta_R(y)|\leq \frac1R|x-y|.
\end{align*}
\end{remark}
Let us fix $R\geq 1$.
Now we will prove the existence and uniqueness of the global solution $u^R$ to the following truncated equation
  \begin{align}
  \begin{split}\label{SEE-trucated-1}
     u(t)=&S_tu_0-\i\int_0^tS_{t-s}(\theta_R(\|u\|_{Y_s})\big(\lambda |u(s)|^{2\sigma}u(s) \big)\d s\\
     &+\int_0^t\int_BS_{t-s}\Big[\Phi(z,u(s-))-u(s-)\Big]\tilde{N}(\d s,\d z)\\
     &+\int_0^t\int_BS_{t-s}\Big[\Phi(z,u(s))-u(s)+\i\sum_{j=1}^mz_jg_j(u(s))\Big]\nu(\d z)\d s,\quad 0\leq t\leq T.
     \end{split}
\end{align}
 For the simplicity of presentation, we shall adopt the following notations for $t\in[0,T]$,
   \begin{align}
     \begin{split}\label{eqn-Psi}
  [\Psi_1^{R}(u)](t)=&-\i\int_0^tS_{t-s}(\theta_R(\|u\|_{Y_s})\big(\lambda |u(s)|^{2\sigma}u(s) \big) \d s,\\
    [\Psi_2^{R}(u)](t) = & \int_0^t\int_BS_{t-s}\Big[\Phi(z,u(s-))-u(s-)\Big]\tilde{N}(\d s,\d z),\\
     [\Psi_3^{R}(u)](t) =&\int_0^t\int_BS_{t-s}\Big[\Phi(z,u(s))-u(s)+\i\sum_{j=1}^mz_jg_j(u(s))\Big]\nu(\d z)\d s.
 \end{split} \end{align}

Now let us estimate the deterministic term $\Psi_1^{R}(u)$.

  \begin{proposition}\label{prop_sto_stri_1}
Assume that $0<\sigma<\frac{2}{n}$ and $r=2\sigma+2$.  Then for every $T>0$ the function $\Psi_1^{R}$ maps   $Y_T$ into itself and,
   for all $u_1,u_2\in Y_T$, we have
  \begin{align}\label{est-phi-1}
     \|\Psi_1^{R}(u_1)-\Psi_1^{R}(u_2)\|_{Y_T}\leq C_{\sigma}|\lambda|R^{2\sigma}T^{1-\frac{n\sigma}{2}}\|u_1-u_2\|_{Y_T}.
  \end{align}
  \end{proposition}
  \begin{proof} Let us choose and fix $T>0$.   Take $u\in Y_{T}$. Let us define $\tau=\inf\{t\geq 0:\|u\|_{Y_t}> 2R\}\wedge T$. Then $\tau$ is a stopping time, c.f. \cite{Bass}. Observe that $\theta_R(\|u\|_{Y_t})=0$ for $\|u\|_{Y_t}\geq 2R$. Since $t\rightarrow \|u\|_{Y_t}$ is non-decreasing on $[0,T]$, we have $\theta_R(|u|_{Y_t})=0$ for $t\geq \tau$.
  By applying the Strichartz inequality \eqref{det-stri-ine-L2} and \eqref{det-stri-ine}, we get
    \begin{align*}
          \sup_{0\leq t\leq T} \|\Psi_1^R(u)(t)\|_{L^2(\RR^n)}
           &\leq C \big\|\theta_R(\|u\|_{Y_\cdot})) \big(\lambda |u|^{2\sigma}u \big)\big\|_{{L^{p'}([0,T];L^{r'}(\RR^n))}};\\
           \|\Psi_1^{R}(u)\|_{L^p(0,T;L^r(\mathbb{R}^n))}&\leq C\|  \theta_R(\|u\|_{Y_\cdot})\big(\lambda |u|^{2\sigma}u \big)\|_{L^{p'}([0,T];L^{r'}(\RR^n))}.
  \end{align*}
  Collecting the above two estimates and then applying H\"older's inequality, we obtain
    \begin{align*}
           \|\Psi_1^R(u)\|_{Y_t}&\leq C\big\|\theta_R(\|u\|_{Y_\cdot})) \big(\lambda |u|^{2\sigma}u \big)\big\|_{{L^{p'}([0,T];L^{r'}(\RR^n))}}\\
  & \leq C|\lambda|\cdot \| |u|^{2\sigma+1}\|_{L^{p'}([0,\tau];L^{r'}(\RR^n))}\\
  &=C|\lambda|\Big(\int_0^{\tau}\Big(\int_{\RR^n}|u(t,x)|^{r}\d x\Big)^{\frac{p'}{r'}}\d t\Big)^{\frac{1}{p'}}\\
  &\leq C|\lambda|\tau^{\frac{p-r}{p-1}\frac{1}{p'}}\Big(\int_0^{\tau}\Big(\int_{\RR^n}|u(t,x)|^{r}\d x\Big)^{\frac{p}{r}}\d t\Big)^{\frac{1}{p}\cdot\frac{r}{r'}}\\
  &\leq C|\lambda|T^{1-\frac{n\sigma}{2}}\Big(\int_0^{\tau}\Big(\int_{\RR^n}|u(t,x)|^{r}\d x\Big)^{\frac{p}{r}}\d t\Big)^{\frac{1}{p}(2\sigma+1)}\\
 &\leq C|\lambda|T^{1-\frac{n\sigma}{2}}\|u\|^{2\sigma+1}_{Y_{\tau}},
  \end{align*}
  where $1-\frac{n\sigma}{2}>0$. Now take $u_1,u_2\in Y_T$. Let us define $\tau_i=\inf\{t\geq 0:|u_i|_{Y_t}> 2R\}\wedge T$, $i=1,2$. Without loss of generality we may assume that $\tau_1\leq \tau_2$.
Similarly, by the Strichartz inequality \eqref{det-stri-ine-L2} and \eqref{det-stri-ine}, we obtain
  \begin{align*}
   &  \sup_{0\leq t\leq T}  \|\Psi_1^{R}(u_1)-\Psi_1^{R}(u_2)\|_{L^2(\mathbb{R}^n)}\leq C\big\|\theta_R(\|u_1\|_{Y_\cdot})\lambda |u_1|^{2\sigma}u_1-\theta_R(\|u_2\|_{Y_\cdot})\lambda |u_2|^{2\sigma}u_2\big\|_{L^{p'}(0,T;L^{r'}(\RR^n))},\\
  & \|\Psi_1^{R}(u_1)-\Psi_1^{R}(u_2)\|_{L^p(0,T;L^r(\mathbb{R}^n))}\leq C\big\|\theta_R(\|u_1\|_{Y_\cdot})\lambda |u_1|^{2\sigma}u_1-\theta_R(\|u_2\|_{Y_\cdot})\lambda |u_2|^{2\sigma}u_2\big\|_{L^{p'}(0,T;L^{r'}(\RR^n))}.
     \end{align*}
   Combing the above two estimates and then applying  Remark \ref{rem-theta},  Taylor's formula and H\"older's inequality, we get
   \begin{align*}
          \|\Psi_1^{R}(u_1)-\Psi_1^{R}(u_2)\|_{{Y_T}}&\leq 2C\big\|\theta_R(\|u_1\|_{Y_\cdot})\lambda |u_1|^{2\sigma}u_1-\theta_R(\|u_2\|_{Y_\cdot})\lambda |u_2|^{2\sigma}u_2\big\|_{L^{p'}(0,T;L^{r'}(\RR^n))}\\
          &\leq 2C\big\Vert \big(\theta_R(\|u_1\|_{Y_\cdot})-\theta_R(\|u_2\|_{Y_\cdot})\big)\lambda |u_2|^{2\sigma}u_2\big\Vert_{L^{p'}(0,T;L^{r'}(\RR^n))}\\
           &\hspace{0.5cm}+2C\big\Vert\theta_R(\|u_1\|_{Y_\cdot})\big(\lambda |u_1|^{2\sigma}u_1-\lambda |u_2|^{2\sigma}u_2\big)\big\|_{L^{p'}(0,T;L^{r'}(\RR^n))}\\
            &\leq 2|\lambda|\frac{C}{R}\|u_1-u_2\|_{Y_T}\left\|u_2\right\|^{2\sigma+1}_{L^{p'}(0,\tau_2;L^{r'}(\RR^n))}\\
           &\hspace{0.5cm}+2|\lambda|C\big\Vert  |u_1|^{2\sigma}u_1- |u_2|^{2\sigma}u_2\big\|_{L^{p'}(0,\tau_1;L^{r'}(\RR^n))}\\
                      &\leq 2|\lambda|\frac{C}{R} T^{1-\frac{n\sigma}{2}}\|u_2\|^{2\sigma+1}_{L^{p}(0,\tau_2;L^{r}(\RR^n))}\|u_1-u_2\|_{Y_T}
          \\&\hspace{0.5cm}+2|\lambda|C_{\sigma}T^{1-\frac{n\sigma}{2}} (\|u_{1}\|_{L^{p}(0,\tau_{1;}L^{r}(\RR^{n}))}+\|u_{2}\|_{L^{p}(0,\tau_{1;}L^{r}(\RR^{n}))})^{2\sigma}\|u_{1}-u_{2}\|_{L^{p}(0,\tau_{1;}L^{r}(\RR^{n}))},
                       \end{align*}
  where  $1-\frac{n\sigma}{2}>0$.
 It follows that
    \begin{align*}
          \|\Psi_1^{R}(u_1)-\Psi_1^{R}(u_2)\|_{Y_{T}}
          \leq 4|\lambda|C T^{1-\frac{n\sigma}{2}}(2R)^{2\sigma}\|u_1-u_2\|_{Y_T}+2|\lambda|C_{\sigma}T^{1-\frac{n\sigma}{2}} (4R)^{2\sigma}\|u_1-u_2\|_{Y_T}
  \end{align*}
  which proves \eqref{est-phi-1}.

  \end{proof}


    To establish the stochastic Strichartz estimates for the stochastic term $\Psi_2^{R}$, we  need  the following two technical lemmas.

    \begin{lemma}\label{lem-norm conservation}
Under Assumption \ref{assu-1}, the Marcus map given in \eqref{phi-equation} satisfies
  \begin{equation}
 \label{eqn-norm conservation}
  |\Phi(\theta,z,y)|=|\Phi(0,z,y)|=|y|, \;\;\;\mbox{ for
all }\theta\in \mathbb{R},\;\; z\in \mathbb{R}^m, \;\;y\in \CC.
\end{equation}
  \end{lemma}
\begin{proof}
 Let us fix $ z\in \mathbb{R}^m$ and $y\in \CC$. Then we have
 \begin{align}
\frac12  \frac{\partial |\Phi(\theta,z,y)|^2}{\partial \theta}&=\Re\big{\langle} \frac{\partial
\Phi(\theta,z,y)}{\partial\theta},\Phi(\theta,z,y) \big{\rangle}\\
&=-\Re\big{\langle} \i\sum_{j=1}^nz_jg_j(\Phi(\theta,z,y)),\Phi(\theta,z,y)\big{\rangle}\\
&=-\sum_{j=1}^n z_j\Re\big{\langle} \i\;\tilde{g}_j(|\Phi(\theta,z,y)|^2)\Phi(\theta,z,y),
\Phi(\theta,z,y)\big{\rangle}\\
&=- \sum_{j=1}^n  \tilde{g}_j(|\Phi(\theta,z,y)|^2)\Re\big[\Phi(\theta,z,y)\overline{\i
\Phi(\theta,z,y)}\big]=0,
  \end{align}
where we used the following fundamental identity:
          \begin{equation}\label{eqn-=0}
                     \Re \langle \i u,u\rangle= \Re \bigl[ \i\langle  u,u\rangle \bigr]=\Re\bigl[ \i \vert u \vert ^2
\bigr]=0,\;\;\ u \in \mathbb{C}.
          \end{equation}
\end{proof}

For the simplicity of notation, for each $s\in[0,T]$, $z\in\mathbb{R}^{m}$, $y\in\mathbb{C}$, we denote
  \begin{align*}
  &\mathcal{G}(s,z,y):=\Phi(s,z,y)-y\\
   &\mathcal{H}(s,z,y):=\Phi(s,z,y)-y+\i\sum_{j=1}^mz_jg_j(y).
  \end{align*}
  For abbreviation, we  let $G(z,y)=\mathcal{G}(1,z,y)$ and $H(z,y)=\mathcal{H}(1,z,y)$.
  \begin{lemma}\label{lem-stoch-lip}
Under Assumption \ref{assu-1}, there exist $C_{m}^{1},C_{m}^{2},C_{m}^{3},C_{m}^{4}>0$  such that for all $y,y_{1},y_{2}\in\mathbb{C}$ and all  $z\in\RR^{m}$: $|z|_{\RR^m}\leq 1$,
 \begin{align}
  |G(z,y)|&\leq C_{m}^{1}|z|_{\RR^m}|y|\label{G-est-growth}\\
   |G(z,y_1)-G(z,y_2)|&\leq C_{m}^{2}|z|_{\RR^m}|y_1-y_2|\label{G-est-Lip}\\
     |H(z,y)|&\leq C_{m}^{3}|z|^2_{\RR^m}|y|\\
   |H(z,y_1)-H(z,y_2)|&\leq C_{m}^{4}|z|^2_{\RR^m}|y_1-y_2|
  \end{align}
  \end{lemma}
 \begin{proof}
 Using the Cauchy-Schwartz inequality  we deduce
 \begin{align*}
 |\mathcal{G}(s,z,y)|=|\Phi(s,z,y)-y|&=\Big|\int_0^s \sum_{j=1}^m -\i z_jg_j(\Phi(a,z,y))\d a\Big|\\
 &\leq  |z|_{\RR^m}\int_0^s \Big(\sum_{j=1}^{m}|g_j(\Phi(a,z,y))|^2\Big)^{\frac12}\d a\\
 &\leq  |z|_{\RR^m}L_{1}m^{\frac12}\int_0^s|\Phi(a,z,y))|\d a\\
  &\leq |z|_{\RR^m}L_{1}m^{\frac12}|y|s+ |z|_{\RR^m}L_{1}m^{\frac12}\int_0^s|\mathcal{G}(a,z,y)|\d a,
 \end{align*}
 where in the last two steps we used assumption \eqref{Lip_of_g} and the definition of $\mathcal{G}$.
Applying the Gronwall lemma yields
  \begin{align}\label{lem-est-eq-10}
     |\mathcal{G}(s,z,y)|&\leq sm^{\frac12}L_{1}|z|_{\RR^m}|y|e^{ m^{\frac12}sL_{1}|z|_{\RR^m}},
  \end{align}
  which proves \eqref{G-est-growth} with $s=1$.
Similarly, we have
  \begin{align}
  \begin{split}\label{lem-est-eq-11}
     |\mathcal{G}(s,z,y_1)-\mathcal{G}(s,z,y_2)|&=|\Phi(s,z,y_1)-y_1-\Phi(s,z,y_2)+y_2|\\
     &=\Big| \int_0^s\sum_{j=1}^m -\i z_j \Big(g_j(\Phi(a,z,y_1))-g_j(\Phi(a,z,y_2))\Big)   \d a    \Big|\\
     &\leq |z|_{\RR^m}\int_0^s \Big(\sum_{j=1}^{m}|g_j(\Phi(a,z,y_1))-g_j(\Phi(a,z,y_2))|^2\Big)^{\frac12}\d a\\
     &\leq m^{\frac12}|z|_{\RR^m}L_{1}\int_0^s|\Phi(a,z,y_1)-\Phi(a,z,y_2)|\d a.
 \end{split}
  \end{align}
  Taking the second and the last expressions of the above inequality we deduce that
  \begin{align*}
         |\Phi(s,z,y_1)-\Phi(s,z,y_2)|\leq |y_1-y_2|+m^{\frac12}|z|_{\RR^m}L_{1}\int_0^s|\Phi(a,z,y_1)-\Phi(a,z,y_2)|^2\d a.
  \end{align*}
  Applying the Gronwall inequality we get
  \begin{align}\label{lem-est-eq-12}
  |\Phi(s,z,y_1)-\Phi(s,z,y_2)|\leq |y_1-y_2|e^{sm^{\frac12}|z|_{\RR^m}L_{1}}.
  \end{align}
 The required result \eqref{G-est-Lip} is obtained on inserting \eqref{lem-est-eq-12} back into \eqref{lem-est-eq-11} and putting $s=1$.
  Observe that
  \begin{align*}
     |H(z,y)|&=|\Phi(1,z,y)-y+\i\sum_{j=1}^mz_jg_j(y) |\\
     &=\Big| \int_0^1  -\sum_{j=1}^m z_j \Big[\i g_j (\Phi(a,l,y))-\i g_j(y)\Big]  \d a\Big| \\
      &=\Big| \int_0^1  \sum_{j=1}^m z_j \int_0^a\Big[\frac{\d (\i g_j)}{\d\Phi}(\Phi(b,z,y))\Big]\Big(-\i\sum_{k=1}^m z_kg_k(\Phi(b,z,y))\Big)\d b  \d a\Big|.
      \end{align*}
      We now apply the Cauchy-Schwartz inequality, Assumption \ref{assu-1} and \eqref{lem-est-eq-10} to obtain
      \begin{align*}
      &|H(z,y)|=|\Phi(1,z,y)-y+\i\sum_{j=1}^mz_jg_j(y) |\\
       & \leq \int_0^1 \int_0^a \Big| \sum_{j=1}^m z_j \Big(\sum_{k=1}^m|z_k|^2\Big)^{\frac12}\Big(\sum_{k=1}^m  \Big|  \tilde{g}_j(|\Phi(b,z,y)|^2)\tilde{g}_k(|\Phi(b,z,y)|^2)\Phi(b,z,y)\Big|^2 \Big)^{\frac12}\d b \d a\\
        &\leq L_{2}|z|^2_{\RR^m}\int_0^1\int_0^a \Big(\sum_{j=1}^m \sum_{i=1}^m  |\Phi(b,z,y)|^2  \Big)^{\frac12}\d b \d a\\
        &\leq mL_{2}|z|^2_{\RR^m}\int_0^1\int_0^a \Big(|y|+|\mathcal{G}(b,z,y)|  \Big)\d b \d a\\
        &=\frac{m}{2}L_{2}|z|^2_{\RR^m}|y|+mL_{2}|z|^2_{\RR^m}\int_0^1\int_0^a bm^{\frac12}L_{1}|z|_{\RR^m}|y|_{\CC}e^{bm^{\frac12}L_{1}|z|_{\RR^{m}}}\d b\d a\\
        &=\frac{m}{2}L_{2}|z|^2_{\RR^m}|y|+Cm^{\frac32}L_{2}L_{1}|z|^2_{\RR^m}|y|\int_0^1\int_0^a be^{m^{\frac12}L_{1}b}\d b\d a\\
        &\leq C_{m}^{3}z|^2_{\RR^m}|y|,
  \end{align*}
  where $C_{m}^{3}=\frac{m}{2}L_{2}+L_{1}L_{2}m^{\frac32}K\int_0^1\int_0^a be^{m^{\frac12}L_{1}b}\d b\d a$.

  A similar argument using assumption \eqref{Lip_of_m_ji} yields
    \begin{align*}
     |H(z,y_1)-H(z,y_2)|
   &=|\Phi(1,z,y_1)-y_1+\i\sum_{j=1}^mz_jg_j(y_1)-\Phi(1,z,y_2)+y_2-\i\sum_{j=1}^mz_jg_j(y_2) |\\
     &=\Big|  \int_0^1 \sum_{j=1}^m z_j \Big[\i g_j (\Phi(a,z,y_1))-\i g_j(y_1)-\i g_{j}(\Phi(a,z,y_2)))+\i g_j(y_2)\Big]  \d a\Big| \\
      &=\Big|\int_0^1 \sum_{j=1}^m z_j \int_0^a\Big[\frac{\d (\i g_j)}{\d\Phi}(\Phi(b,z,y_1))(-\i)\sum_{k=1}^m z_kg_k(\Phi(b,z,y_1))\\
      &\hspace{3cm}-\frac{\d (\i g_j)}{\d\Phi}(\Phi(b,z,y_2))(-\i)\sum_{k=1}^m z_kg_k(\Phi(b,z,y_2))\Big]\d b  \d a\Big|\\
      &\leq \int_0^1 \int_0^a \Big| \sum_{j=1}^m z_j \sum_{k=1}^m z_k\Big[\frac{\d (\i g_j)}{\d\Phi}(\Phi(b,z,y_1))(-\i g_k)(\Phi(b,z,y_1))\\
      &\hspace{4cm} -\frac{\d (\i g_j)}{\d\Phi}(\Phi(b,z,y_2))(-\i g_k)(\Phi(b,z,y_2)) \Big]\Big|\d b \d a\\
         &\leq |z|^2_{\RR^m}\int_0^1\int_0^a \Big(\sum_{j=1}^m \sum_{k=1}^m \Big|\frac{\d (\i g_j)}{\d\Phi}(\Phi(b,z,y_1))(-\i)\sum_{k=1}^m z_kg_k(\Phi(b,z,y_1))\\
         &-\frac{\d (\i g_j)}{\d\Phi}(\Phi(b,z,y_2))(-\i)\sum_{k=1}^m z_kg_k(\Phi(b,z,y_2))\Big)\Big|^2  \Big)^{\frac12}\d b \d a\\
        &\leq L_{2}|z|^2_{\RR^m}\int_0^1\int_0^a \Big(\sum_{j=1}^m \sum_{k=1}^m  \big|\Phi(b,z,y_1)-\Phi(b,z,y_2)\big|^2  \Big)^{\frac12}\d b \d a\\
       &\leq mL_{2}|z|^2_{\RR^m}\int_0^1\int_0^a  |y_1-y_2|e^{bm^{\frac12}|z|_{\RR^m}L_{1}}\d b \d a\\
       &\leq C_{m}^{4}|z|^2_{\RR^m}|y_1-y_2|,
  \end{align*}
  where $C_{m}^{4}=mL_{2}\int_0^1\int_0^a  e^{bm^{\frac12}L_{1}}\d b \d a$ and  the proof of Lemma \ref{lem-stoch-lip} is concluded.
   \end{proof}

  \begin{proposition}\label{prop_sto_stri_2} Under Assumption \ref{assu-1}, for every $T>0$ the function $\Psi_2^{R}$ maps $M^p_{\mathbb{F}}(Y_T)$ into itself and for all $u_{1},u_{2}\in Y_{T}$ we have
  \begin{align*}
   \EE\|\Psi_2^{R}(u_1)-\Psi_2^{R}(u_2)\|_{Y_T}^p
 &\leq C_{p,m}(T^{\frac{p}{2}}+T)\EE\|u_1-u_2\|^p_{Y_T}.
  \end{align*}

  \end{proposition}

  \begin{proof}  Let us choose and fix $T>0$.   Take $u\in Y_{T}$.
By using  Proposition \ref{prop-str-ineq-2} and Lemma \ref{lem-stoch-lip} we have
  \begin{align*}
  \EE\sup_{t\in[0,T]}\|\Psi_2^{R}(u)(t)\|^p_{L^2(\RR^n)}
  &\leq C_p\, \EE\Big(\int_0^T \int_B \|\Phi(z,u(s))-u(s)\|_{L^2(\RR^{n})}^2\nu(\d z)\d s\Big)^{\frac{p}{2}}\\
  &\quad+C_p\, \EE\Big(\int_0^T \int_B\|\Phi(z,u(s))-u(s)\|_{L^2(\RR^{n})}^p\nu(\d z)\d s\Big)\\
  &\leq C_{p,m}\, \EE\Big(\int_0^{\tau}\int_B|z|^2_{\RR^m} \|u(s)\|_{L^2(\RR^{n})}^2\nu(\d z)\d s\Big)^{\frac{p}{2}}\\
  &\quad+C_{p,m}\, \EE\Big(\int_0^{\tau}\int_B|z|^p_{\RR^m} \|u(s)\|^p_{L^2(\RR^{n})}\nu(\d z)\d s\Big)\\
   & \leq C_{p,m}(T^{\frac{p}{2}}+T)\E\|u\|^p_{Y_{T}}\Big(\Big(\int_B|z|^p_{\RR^m}\nu(\d z)\Big)^{\frac{p}{2}}+\int_B|z|^2_{\RR^m}\nu(\d z) \Big)\\
    &\leq  C_{p,m}(T^{\frac{p}{2}}+T) \E\|u\|^p_{Y_{T}}
  \end{align*}
  from which we also deduce that $\Psi_2^{R}(u)$ is $L^{2}(\RR^{n})$-valued adapted with c\`adl\`ag modification, see e.g.\cite[Theorem 3.1]{Br+Liu+Zhu2019}.
Here we also used the fact that $\int_B|z|^p_{\RR^m}\nu(\d z) \leq \int_B|z|^2_{\RR^m}\nu(\d z)<\infty$ for $p\geq 2$.

By applying Proposition \ref{prop-str-ineq-11} and Lemma \ref{lem-stoch-lip} we get
      \begin{align*}
  \EE\|\Psi_2^{R}(u)\|_{L^p(0,T;L^r(\mathbb{R}^n))}^p
  &\leq C_{p}\, \EE\Big(\int_0^T \int_B \|\Phi(z,u(s))-u(s)\|_{L^2(\RR^{n})}^2\nu(\d z)\d s\Big)^{\frac{p}{2}}\\
  &\quad+C_{p}\,\EE\Big(\int_0^T \int_B\|\Phi(z,u(s))-u(s)\|_{L^2(\RR^{n})}^p\nu(\d z)\d s\Big)\\
  &\leq C_{p,m}\EE\Big(\int_0^{T}\int_{B}|z|_{\RR^m}^2|u(s)|^2_{L^2(\RR^n)}\nu(\d z)\d t\Big)^{\frac{p}{2}}\\
  &\quad+C_{p,m}\EE\Big(\int_0^{T}\int_{B}|z|_{\RR^m}^p|u(s)|^p_{L^2(\RR^n)}\nu(\d z)\d t\Big)\\
    &\leq C_{p,m}(T^{\frac{p}{2}}+T)\EE\|u\|^p_{Y_{T}}.
  \end{align*}
Thus we infer that
  \begin{align*}
 \EE\|\Psi_2^{R}(u)\|^p_{Y_T}
 \leq C_{p,m}(T^{\frac{p}{2}}+T)\EE\|u\|^p_{Y_{T}}.
  \end{align*}
   Take now $u_1,u_2\in Y_T$. By Propositions \ref{prop-str-ineq-2} and  \ref{prop-str-ineq-11}, Lemma \ref{lem-stoch-lip}, applying similar arguments as before we obtain
   \begin{align*}
 &  \EE\sup_{t\in[0,T]}\|\Psi_2^{R}(u_1)(t)-\Psi_2^{R}(u_2)(t)\|^p_{L^2(\RR^n)}\\
   \leq &C_{p}\,\EE\Big(\int_0^T\int_{B}\|G(z,u_{1}(s-))-G(z,u_{2}(s-))\|_{L^2(\RR^n)}^{2}\nu(\d z)\d s\Big)^{\frac{p}2}\\
  &+C_{p}\,\EE\int_0^T\int_{B}\|G(z,u_{2}(s-))-G(z,u_{2}(s-))\|_{L^2(\RR^n)}^{p}\nu(\d z)\d s\\
    \leq& C_{p,m}(T^{\frac{p}{2}}+T)  \E\|u_1-u_2\|^p_{L^{\infty}([0,T];L^2(\RR^n))}\Big(\Big(\int_B|z|^2_{\RR^m}\nu(\d l)\Big)^{\frac{p}{2}}+\int_B|z|^p_{\RR^m}\nu(\d z) \Big)\\
    \leq&  C_{p,m}(T^{\frac{p}{2}}+T) \E\|u_1-u_2\|^p_{Y_{T}},
  \end{align*}
and
    \begin{align*}
  &\EE\|\Psi_2^{R}(u_1)-\Psi_2^{R}(u_2)\|_{L^p(0,T;L^r(\mathbb{R}^n))}^p\\
 &\leq C_{p}\, \EE\Big(\int_0^T\int_{Z}\|G(z,u_1(s))-\theta_R(\|u_{2}\|_{Y_s})G(z,u_2(s))\|^2_{L^2(\RR^n)}\nu(\d z)\d t\Big)^{\frac{p}{2}}\\
  &\quad+C_{p}\,\EE\Big(\int_0^T\int_{Z}\|G(z,u_1(s))-\theta_R(\|u_{2}\|_{Y_s})G(z,u_2(s))\|^p_{L^2(\RR^n)}\nu(\d z)\d t\Big)\\
  &\leq  C_{p,m}(T^{\frac{p}{2}}+T) \E\|u_1-u_2\|^p_{Y_{T}}.  \end{align*}
  Combining the above estimates gives that
      \begin{align*}
\EE\|\Psi_2^{R}(u_1)-\Psi_2^{R}(u_2)\|_{Y_T}^p
 \leq C_{p,m}(T^{\frac{p}{2}}+T)\EE\|u_1-u_2\|^p_{Y_T}.
  \end{align*}
  This concludes the proof of Proposition \ref{prop_sto_stri_2}.
  \end{proof}
  \begin{proposition}\label{prop_sto_stri_3} Suppose that Assumption \ref{assu-1} hold. Then for every $T>0$ the function $\Psi_3^{R}(u)$ maps $Y_T$  into itself and for all $u_{1},u_{2}\in Y_{T}$ we have
  \begin{align*}
    \|\Psi_3^{R}(u_{1})-\Psi_3^{R}(u_{2})\|_{Y_T}
&\leq C_{m}T \|u_1-u_2\|_{Y_T}.
  \end{align*}

  \end{proposition}

  \begin{proof}  Let us choose and fix $T>0$.   Take $u\in Y_{T}$.   Observe that
  \begin{align*}
  \sup_{0\leq t\leq T}\|\Psi_3^{R}(u)(t)\|_{L^2(\mathbb{R}^n)}&=\sup_{t\in[0,T]}\Big\|\int_0^t\int_BS_{t-s} H(z,u(s))
  \nu(\d l)\d s\Big\|_{L^2(\RR^n)}\\
  &\leq C \int_0^{T}\int_B\|H(z,u(s))\|_{L^2(\RR^n)}\nu(\d z)\d s\\
  &\leq C_{m}\int_0^{T}\Big(\int_B|z|^2_{\RR^m}\nu(\d z)\Big)\|u(s)\|_{L^2(\RR^n)}\d s\\
  &\leq C_{m}T \|u\|_{Y_{T}}\Big(\int_B|z|^2_{\RR^m}\nu(\d z)\Big).
  \end{align*}
  By applying Remark \ref{remark-stri-eq-1} and Lemma \ref{lem-stoch-lip} we obtain
    \begin{align*}
  \|\Psi_3^{R}(u)\|_{L^p(0,T;L^r(\mathbb{R}^n))}&=\Big\|\int_0^t\int_BS_{t-s}H(z,u(s))\nu(\d z)\d s\Big\|_{L^p(0,T;L^r(\RR^n))}\\
  &=\Big\|\int_0^tS_{t-s}\Big(\int_BH(z,u(s))\nu(\d z)\Big)\d s\Big\|_{L^p(0,T;L^r(\RR^n))}\\
&\leq C\int_0^{T}\Big{\|}\int_BH(z,u(s))\nu(\d z)\Big{\|}_{L^2(\RR^n)}\d s\\
&\leq C_{m}\int_0^{T}\Big(\int_B|z|^2_{\RR^m}\nu(\d z)\Big)\|u(s)\|_{L^2(\RR^n)}\d s\\
&\leq C_{m}T \|u\|_{Y_{T}}\Big(\int_B|z|^2_{\RR^m}\nu(\d z)\Big).
  \end{align*}
  It follows that
  \begin{align*}
  \|\Psi_3^{R}(u)\|_{Y_T}\leq C_{m}T\|u\|_{Y_T}.
  \end{align*}
  Take now $u_1,u_2\in Y_T$.    Again by using Remark \ref{remark-stri-eq-1} and Lemma \ref{lem-stoch-lip}, we can show that
    \begin{align*}
&\sup_{0\leq t\leq T} \|\Psi_3^{R}(u_1)(t)-\Psi_3^{R}(u_2)(t)\|_{L^2(\mathbb{R}^n)}\\
  &\leq C \int_0^T\int_B\|H(z,u_1(s))-H(z,u_2(s))\|_{L^2(\RR^n)}\nu(\d z)\d s\\
  &\leq C_{m}\int_0^{T}\Big(\int_B|z|^2_{\RR^m}\nu(\d z)\Big)\|u_1(s)-u_2(s)\|_{L^2(\RR^n)}\d s\\
  &\leq C_{m}T \|u_1-u_2\|_{Y_{T}},
    \end{align*}
and
    \begin{align*}
 & \|\Psi_3^{R}(u_1)-\Psi_3^{R}(u_2)\|_{L^p(0,T;L^r(\mathbb{R}^n))}\\
  &=\Big\|\int_0^tS_{t-s}\Big(\int_B H(z,u_1(s))-H(z,u_2(s))\nu(\d z)\Big)\d s\Big\|_{L^p(0,T;L^r(\RR^n))}\\
&\leq C\int_0^T\Big\|\int_B H(z,u_1(s))-H(l,u_2(s))\nu(\d z)\Big\|_{L^2(\RR^n)}\d s\\
&\leq C_{m}\int_0^T\Big(\int_B|z|^2_{\RR^m}\nu(\d z)\Big)^{\frac12}\|u_1(s)-u_2(s)\|_{L^2(\RR^n)}\d s\\
&\leq C_{m}T\|u_1-u_2\|_{Y_{T}}.
  \end{align*}
  Combining the above estimates gives that
      \begin{align*}
  \|\Psi_3^{R}(u_{1})-\Psi_3^{R}(u_{2})\|_{Y_T}
&\leq C_{m}T \|u_1-u_2\|_{Y_T}.
  \end{align*}
The proof of Proposition \ref{prop_sto_stri_3} is now complete.
  \end{proof}

We are now ready to present the main result in this section.
  \begin{proposition}\label{prop-truncated equation} Let $p\geq 2$, $0<\sigma<\frac{2}{n}$, $u_0\in L^p(\Omega;L^2(\RR^n))$ and $(p,r)$ be an admissible pair with $r=2\sigma+2$.
    Under Assumption \ref{assu-1}, for every $T>0$ there exists a unique global solution $u^R$ in $L^p\big(\Omega;L^{\infty}(0,T;L^2(\RR^n))\cap L^p(0,T;L^r(\RR^n))\big)$ to equation \eqref{SEE-trucated-1}.
  \end{proposition}

  \begin{proof}We will carry this out in two stages.

 \emph{Step 1.}
Define an operator by
  \begin{align*}
       \Gamma^R(u)(t):=S_tu_0+\Psi_1^{R}(u)(t)+\Psi_2^{R}(u)(t)+\Psi_3^{R}(u)(t),\;\; t\in[0,T].
  \end{align*}
  We will construct a unique solution by the Banach fixed point theorem. Combining Remark \ref{remark-stri-eq-1},  Proposition \ref{prop_sto_stri_1}, \ref{prop_sto_stri_2} and \ref{prop_sto_stri_3}, we see that for every $T$, the operator $\Gamma^R$ maps from $M^p_{\mathbb{F}}(Y_T)$ into $M^p_{\mathbb{F}}(Y_T)$.
 Now let us show that  if  $T$ is sufficiently small, which will be determined later, then this operator is a strict contraction in the space $M_{\mathbb{F}}^p(Y_T)$. It follows again from  Proposition \ref{prop_sto_stri_1}, \ref{prop_sto_stri_2} and \ref{prop_sto_stri_3} that
 \begin{align*}
 \|\Gamma^R(u_1)-\Gamma^R(u_2)\|_{M^p_{\mathbb{F}}(Y_T)}\leq \Big[  C_{\sigma}|\lambda|R^{2\sigma}T^{1-\frac{n\sigma}{2}}+ C_{p,m}(T^{\frac{p}{2}}+T)   +C_{m}  T \Big] \|u_1-u_2\|_{M^p_{\mathbb{F}}(Y_T)}.
 \end{align*}
     If we choose $T_{0}$ sufficiently small (depending on $R,\sigma,p,m$) such that
     \begin{align*}
        C_{\sigma}|\lambda|R^{2\sigma}T_{0}^{1-\frac{n\sigma}{2}}+ C_{p,m}(T_{0}^{\frac{p}{2}}+T_{0})   +C_{m} T_{0} \leq \frac12,
     \end{align*}
    then $\Gamma^R$ is a $\frac12$-contraction in the space $M^p_{\mathbb{F}}(Y_{T_{0}})$. Hence by the Banach fixed point theorem, there exists a unique solution $u^R\in M^p_{\mathbb{F}}(Y_{T_{0}})$ satisfying $u^R=\Gamma^R(u^R)$. Note that one can always find a c\`adl\`ag modification of $\Psi_2^{R}(u^{R})$ in $L^{2}(\RR^{n})$, see e.g. \cite{Zhu+Brz+Hau2017,Br+Liu+Zhu2019}. We will identify $\Psi_2^{R}(u^{R})$ with this modification. It follows that the solution $u^{R}$ is c\`adl\`ag in $L^{2}(\RR^{n})$. This is the unique solution in $M^p_{\mathbb{F}}(Y_{T_{0}})$ of equation \eqref{SEE-trucated-1} restricted to $[0,T_{0}]$.

     \emph{Step 2.} We will extend the solution to $[0,T]$ by induction, c.f. \cite{Brz+Peng+Zhai_2021}.
     Define $j=[\frac{T}{T_{0}}+1]$. Assume that for some $k\in\{1,2,\cdots,j\}$ there exists $u^R_{k}\in M^p_{\mathbb{F}}(Y_{T_{0}})$ such that
     \begin{align*}
     u^R_{k}=\Gamma^R(u_{k}^R) \text{   on }[0,kT_{0}].
     \end{align*}
     First let us define a new cutoff function by
     \begin{align*}
     \Theta^{R}_{k}(u)(t):=\theta_R(\phi(u)(t)),
     \end{align*}
     where
     \begin{align*}
     \phi(u)(t)=&\Big(\|u_{k}^{R}\|^{p}_{L^{p}(0,kT_{0};L^{r}(\RR^{n}))}+\|u\|^{p}_{L^{p}(0,t;L^{r}(\RR^{n}))}\Big)^{\frac1p}\\
     &+\max\big{\{}\sup\limits_{0\leq t\leq kT_0} \|u_{k}^{R}(t)\|_{L^{2}(\RR^{n})},\sup\limits_{0\leq s\leq t} \|u(s)\|_{L^{2}(\RR^{n})}\big{\}}.
     \end{align*}
    Consider the following operator
     \begin{align}
         \Gamma^R_{k}(u)(t):= S_tu_k^{R}(kT_{0})+\Psi_1^{R,k}(u)(t)+\Psi_2^{R,k}(u)(t)+\Psi_3^{R,k}(u)(t),\;\; u\in M^{p}_{\mathbb{F}^{kT_{0}}}(Y_{T_{0}}).
     \end{align}
     where for $t\in[0,T_{0}]$ and $u\in M^{p}_{\mathbb{F}^{kT_{0}}}(Y_{T_{0}})$,
       \begin{align*}
  [\Psi_1^{R,k}(u)](t)=&-\i\int_0^tS_{t-s}(\Theta^{R}_{k}(u)(s))\big(\lambda |u(s)|^{2\sigma}u(s)\big)\d s,\\
    [\Psi_2^{R,k}(u)](t) = & \int_0^t\int_BS_{t-s}\big[\Phi(z,u_{s-})-u_{s-}\big]\tilde{N}^{kT_{0}}(\d s,\d l),\\
     [\Psi_3^{R,k}(u)](t) =&\int_0^t\int_BS_{t-s}\big[\Phi(z,u_{s})-u_{s}+\i\sum_{j=1}^mz_jg_j(u_s)\big]\nu(\d z)\d s.
  \end{align*}
  All the arguments in Step 1 can be reproduced. Take $v_1,v_2\in M^{p}_{\mathbb{F}^{kT_{0}}}(Y_{T_{0}})$. Let us define $\tau_i=\inf\{t\geq 0:\phi(v_i)(t)> 2R\}\wedge T_{0}$, $i=1,2$. Without loss of generality we may assume that $\tau_1\leq \tau_2$.
By following the same line of argument as used in the proof of Proposition \ref{prop_sto_stri_1}, we infer that
    \begin{align*}
         & \|\Psi_1^{R,k}(v_1)-\Psi_1^{R,k}(v_2)\|_{{Y_{T_{0}}}}\\
          &\leq 2|\lambda|C\Big\|\Theta^{R}_{k}(v_{1})(s)|v_1(s)|^{2\sigma}v_1(s)-\Theta^{R}_{k}(v_{2})(s)|v_2(s)|^{2\sigma}v_2(s)\Big\|_{L^{p'}(0,T;L^{r'}(\RR^n))}\\
          &\leq 2|\lambda|C\Big\Vert \big(\Theta^{R}_{k}(v_{1})(s)-\Theta^{R}_{k}(v_{2})(s))\big)|v_2(s)|^{2\sigma}v_2(s)\Big\Vert_{L^{p'}(0,\tau_2;L^{r'}(\RR^n))}\\
           &\hspace{0.5cm}+2|\lambda|C\Big\Vert\Theta^{R}_{k}(v_{1})(s)\big(|v_1(s)|^{2\sigma}v_1(s)-|v_2(s)|^{2\sigma}v_2(s)\big)\Big\|_{L^{p'}(0,\tau_1;L^{r'}(\RR^n))}\\
           &\leq 2|\lambda|\frac{C}{R}T_{0}^{1-\frac{n\sigma}{2}}\|v_1-v_2\|_{Y_{T_{0}}}\left\|v_2\right\|^{2\sigma+1}_{L^{p}(0,\tau_2;L^{r}(\RR^n))}\\
           &+2|\lambda|C_{\sigma}T_{0}^{1-\frac{n\sigma}{2}}\left(\|v_1\|_{L^{p}(0,\tau_1;L^{r}(\RR^n))}+\|v_2\|_{L^{p}(0,\tau_1;L^{r}(\RR^n))}\right)^{2\sigma}\|v_1-v_2\|_{L^{p}(0,\tau_1;L^{r}(\RR^n))}\\
          & \leq 4|\lambda|C T_{0}^{1-\frac{n\sigma}{2}}(2R)^{2\sigma}\|v_1-v_2\|_{Y_{T_{0}}}+2|\lambda|C_{\sigma}T_{0}^{1-\frac{n\sigma}{2}} (4R)^{2\sigma}\|v_1-v_2\|_{Y_{T_{0}}}\\
          &\leq  C_{\sigma}|\lambda|R^{2\sigma}T_{0}^{1-\frac{n\sigma}{2}}\|v_1-v_2\|_{Y_{T_{0}}}.
  \end{align*}
 Now using the same argument as in the proof of Proposition \ref{prop_sto_stri_2} and \ref{prop_sto_stri_3} and applying Corollary \ref{prop-str-ineq-1}, we obtain
  \begin{align}
   \E \|\Psi_2^{R,k}(v_1)-\Psi_2^{R,k}(v_2)\|_{{Y_{T_{0}}}}^{p}&\leq C_{p,m}(T^{\frac{p}{2}}+T)\EE\|v_1-v_2\|^p_{Y_{T_{0}}}\\
        \|\Psi_3^{R,k}(v_1)-\Psi_3^{R,k}(v_2)\|_{{Y_{T_{0}}}}&\leq C_{m} T\|v_1-v_2\|^p_{Y_{T_{0}}}.
  \end{align}
Hence we conclude that
 \begin{align*}
 &\|\Gamma^R_{k}(v_1)-\Gamma^R_{k}(v_2)\|_{M^p_{\mathbb{F}^{kT_{0}}}(Y_{T_{0}})}\\&\leq \Big[    C_{\sigma}|\lambda|R^{2\sigma}T_{0}^{1-\frac{n\sigma}{2}}+ C_{p,m}(T_{0}^{\frac{p}{2}}+T_{0})   +C_{m} T_{0}  \Big] \|v_1-v_2\|_{M^p_{\mathbb{F}^{kT_{0}}}(Y_{T_{0}})}.
 \end{align*}
Note that the constant is the same as in the first step. It follows that $\Gamma^R_{k}$ is a $\frac12$-contraction in the space $M^p_{\mathbb{F}^{kT_{0}}}(Y_{T_{0}})$.  Let $v_{k+1}^{R}$ be the unique solution satisfying $ v^R_{k+1}=\Gamma^R_{k}(v_{k+1}^R)$.
Then we construct a solution as follows
\begin{align*}
    u_{k+1}^{R}(t)=\left\{
    \begin{array}{rcl} u_{k}^{R}(t), & \mbox{for} & t\in[0,kT_{0}]\\
    v_{k+1}^{R}(t-kT_{0}), & \mbox{for} & t\in [kT_{0},(k+1)T_{0}]
    \end{array}\right.
\end{align*}
and so on, recursively. Notice that $ u_{k+1}^{R}$ is $\mathbb{F}$-adapted, c\`{a}dl\`ag in $L^{2}(\RR^{n})$ and we have $\E\|u_{k+1}^{R}\|^{p}_{Y_{(k+1)T_{0}}}<\infty.$ Therefore, we obtain that $u_{k+1}^{R}\in M^p_{\mathbb{F}}(Y_{(k+1)T_{0}})$.

Now we shall show that $u_{k+1}^{R}$ is a fixed point of $\Gamma^R$ in $M^p_{\mathbb{F}}(Y_{(k+1)T_{0}})$. Let $t\in [kT_{0}, (k+1)T_{0}]$. Define $\hat{t}:=t-kT_{0}$. Then we have
\begin{align*}
   u_{k+1}^{R}(t)&=v_{k+1}^{R}(\hat{t})=\Gamma^R_{k}(v_{k+1}^R)(\hat{t})\\
   &=S_{\hat t}u_k^{R}(kT_{0})+\Psi_1^{R,k}(v_{k+1}^R)(\hat{t})+\Psi_2^{R,k}(v_{k+1}^R)(\hat{t})+\Psi_3^{R,k}(v_{k+1}^R)(\hat{t})\\
   &=S_{\hat t}S_{kT_{0}}u_{0}+S_{\hat t}\Psi_{1}^{R}(u_{k}^{R})(kT_{0})+\Psi_1^{R,k}(v_{k+1}^R)(\hat{t})+S_{\hat t}\Psi_{2}^{R}(u_{k}^{R})(kT_{0})+\Psi_2^{R,k}(v_{k+1}^R)(\hat{t})\\
   &\hspace{1cm}+S_{\hat t}\Psi_{3}^{R}(u_{k}^{R})(kT_{0})+\Psi_3^{R,k}(v_{k+1}^R)(\hat{t}).
\end{align*}
Observe that $\theta_{R}(\|u_{k}^{R}\|_{Y_{s}})=\theta(\|u_{k+1}^{R}\|_{Y_{s}})$ for $s\in [0,kT_{0}]$ and $\Theta_{k}^{R}(v_{k+1}^{R}(s))=\theta_{R}(\|u_{k+1}^{R}\|_{Y_{kT_{0}+s}})$ for $s\in [0,T_{0}]$. It follows that
\begin{align*}
        S_{\hat t}\Psi_{1}^{R}(u_{k}^{R})(kT_{0})+\Psi_1^{R,k}(v_{k+1}^R)(\hat{t})=&-\i S_{\hat t}\int_0^{kT_{0}}S_{kT_{0}-s}\big((\theta_R(\|u_{k}^{R}\|_{Y_s}))\lambda |u_{k}^{R}(s)|^{2\sigma}u_{k}^{R}(s)\big)\d s\\
        &-\i\int_0^{\hat{t}}S_{\hat t-s}\big(\Theta^{R}_{k}(v_{k+1}^{R})(s))\lambda |v_{k+1}^{R}(s)|^{2\sigma}v_{k+1}^{R}(s)\big)\d s\\
        =&-\i\int_{0}^{kT_{0}}S_{t-s}\big(\theta_{R}(\|u_{k+1}^{R}\|_{Y_{s}})\lambda |u_{k+1}^{R}(s)|^{2\sigma}u_{k+1}^{R}(s)\big)\d s\\
        &-\i\int_0^{\hat{t}}S_{\hat t-s}\big(\theta_{R}(\|u_{k+1}^{R}\|_{Y_{kT_{0}+s}})\lambda |u_{k+1}^{R}(kT_{0}+s)|^{2\sigma}u_{k+1}^{R}(kT_{0}+s)\big)\d s\\
        =&- \i\int_{0}^{t}S_{t-s}\big(\theta_{R}(\|u_{k+1}^{R}\|_{Y_{s}})\lambda |u_{k+1}^{R}(s)|^{2\sigma}u_{k+1}^{R}(s)\big)\d s\\
        =&=\Psi_1^{R}(u_{k+1}^R)(t).
\end{align*}
Similarly, by using \eqref{changle-variable-SC}, we can prove
\begin{align*}
    & S_{\hat t}\Psi_{2}^{R}(u_{k}^{R})(kT_{0})+\Psi_2^{R,k}(v_{k+1}^R)(\hat{t})=\Psi_2^{R}(u_{k+1}^R)(t)\\
    & S_{\hat t}\Psi_{3}^{R}(u_{k}^{R})(kT_{0})+\Psi_3^{R,k}(v_{k+1}^R)(\hat{t})=\Psi_2^{R}(u_{k+1}^R)(t).
\end{align*}
Therefore, we infer for $t\in [kT_{0}, (k+1)T_{0}]$
\begin{align*}
u_{k+1}^{R}(t)=S_{t-s}u_{0}+\Psi_1^{R}(u_{k+1}^R)(t)+\Psi_2^{R}(u_{k+1}^R)(t)+\Psi_3^{R}(u_{k+1}^R)(t),
\end{align*}
which shows that $u_{k+1}^{R}$ is a fixed point of $\Gamma^R$ in $M^p_{\mathbb{F}}(Y_{(k+1)T_{0}})$.  Therefore $u^{R}:=u^{R}_{j}$ is the unique solution to \eqref{SEE-trucated-1} on $[0,T]$.
  \end{proof}

By using the above results for the truncated problem \eqref{SEE-trucated-1}, we can derive the existence and the uniqueness of local mild solutions for the original equation \eqref{NSS-integral-form}. The following arguments are standard. One can also see \cite[Proposition 1]{Brz+Zhu2016} for analogous arguments of proving the existence result for maximal local mild solutions to stochastic nonlinear beam equations.
  \begin{proposition}\label{prop-local-mind-main}
     For each $k\in\mathbb{N}$, let $u_k\in M^p_{\mathbb{F}}(Y_T)$ be the solution of \eqref{SEE-trucated-1} provided by Proposition \ref{prop-truncated equation} with $R$ replaced by $k$. Define a stopping time $\tau_k$ by
     \begin{align}\label{eqn-tau_k}
        \tau_k=\inf\{t\in[0,T]:\|u^{k}\|_{Y_{t}}> k\},
     \end{align}
     with the usual convention $\inf \emptyset =T$.
       Then
          \begin{enumerate}
          \item[(1)] For $k\leq n$, we have $0<\tau_k\leq \tau_n$, $\mathbb{P}$-a.s. and $u^{k}(t)=u^{n}(t)$ $\mathbb{P}$-a.s. for $t\in[0,\tau_k]$.
       \item[(2)] Define $u(t)=u^{k}(t)$ for $t\in[0,\tau_k]$ and $\tau_{\infty}=\lim_{n\rightarrow\infty}\tau_k$. Then $(u(t))_{t\in[0,\tau_{\infty})}$ is a  maximal local mild solution of \eqref{NSS-integral-form}.
       \item[(3)] The solution $u$ is unique.
          \end{enumerate}
  \end{proposition}

  \begin{proof}
  For any $n\in\mathbb{N}$, by Proposition \ref{prop-truncated equation}, there exists a unique global solution $u^n$ in $M^p_{\mathbb{F}}(Y_T)$ to equation \eqref{SEE-trucated-1} which satisfies
    \begin{align}
    \begin{split}
  &   u^{n}(t)=S_tu_0-\i\int_0^tS_{t-s}\big(\theta_n(\|u^{n}\|_{Y_s})\lambda |u^{n}(s)|^{2\sigma}u^{n}(s)\big)\d s\\
     &+\int_0^t\int_BS_{t-s}\Big[\Phi(z,u^{n}(s-))-u^{n}(s-)\Big]\tilde{N}(\d s,\d z)\\
     &+\int_0^t\int_BS_{t-s}\Big[\Phi(z,u^{n}(s))-u^{n}(s)+\i\sum_{j=1}^mz_jg_j(u^{n}(s))\Big]\nu(\d z)\d s, \quad \mathbb{P}\text{-a.s.}
     \end{split}
\end{align}
for $t\in[0,T]$.
 For $k\leq n$, set $\tau_{k,n}=\tau_k\wedge\tau_n$.  Hence by the definition of $\theta_{n}$,  we have $\theta_{n}(\|u^{n}\|_{Y_{t}})=1$ and  $\theta_{k}(\|u^{k}\|_{Y_{t}})=1$, for $t\in[0,\tau_{k,n})$. It follows that on $[0,\tau_{k,n})$ we have
     \begin{align}
     \begin{split}\label{prop-local-eq1}
     &u^{l}(t)=S_tu_0-\i\int_0^tS_{t-s}\big(\lambda|u^{l}(s)|^{2\sigma}u^{l}(s)\big)\d s+\int_0^t\int_BS_{t-s}\Big[\Phi(z,u^{l}(s-))-u^{l}(s-)\Big]\tilde{N}(\d s,\d z)\\
     &+\int_0^t\int_BS_{t-s}\Big[\Phi(z,u^{l}(s))-u^{l}(s)+\i\sum_{j=1}^mz_jg_j(u^{l}(s))\Big]\nu(\d z)\d s, \quad \mathbb{P}\text{-a.s. for }l=k,n.
     \end{split}
\end{align}
The uniqueness of the solution to equation \eqref{SEE-trucated-1} implies that  $u^{k}(t)=u^{n}(t)$ $\mathbb{P}$-a.s. on $\{t<\tau_{k,n}\}$.
 Let $\Phi(u^{l})$ denote the right hand side of \eqref{prop-local-eq1}.
Note that the value of $\Phi(u^{l})$ at $\tau_{k,n}$ depends only on the values of $u^{l}$ on $[0,\tau_{k,n})$.
  Hence we may extend the process $u^l$ from the interval  $[0,\tau_{k,n})$ to the closed interval  $[0,\tau_{k,n}]$ by setting
\begin{align}
\begin{split}\label{eq-250}
 &u^{l}(\tau_{k,n})=\Phi(u^{l})(\tau_{k,n})=S_{\tau_{k,n}}u_0-\i\int_0^{\tau_{k,n}}S_{\tau_{k,n}-s}\big(\lambda |(u^{l} (s)|^{2\sigma}(u^{l} (s)\big)\d s+I_{\tau_{k,n}}(\Phi(z,u^{l})-u^{l})(\tau_{k,n})\\
     &+\int_0^{\tau_{k,n}}\int_BS_{\tau_{k,n}-s}\Big[\Phi(z,u^{l}(s))-u^{l}(s)+\i\sum_{j=1}^mz_jg_j(u^{l}(s))\Big]\nu(\d z)\d s, \quad \mathbb{P}\text{-a.s.}
     \end{split}
\end{align}
where $I_{\tau_{k,n}}(\Phi(z,u)-u)$ is a process defined by
          \begin{align*}
         I_{\tau_{k,n}}(\Phi(z,u^{l})-u^{l})(t)= \int_0^t\int_B1_{[0,\tau_{k,n}]}S_{t-s}\Big[\Phi(z,u^{l}(s\wedge\tau_{k,n}-))-u^{l}(s\wedge\tau_{k,n}-)\Big]\tilde{N}(\d s,\d z),
          \end{align*}
          for $0  \leq t\leq T$.  Therefore,  combining the above two equalities \eqref{prop-local-eq1} and \eqref{eq-250}, we deduce that the stopped process $u^{l}(\cdot\wedge\tau_n)$ satisfies
       \begin{align}
     \begin{split}\label{prop-local-eq1-1}
     u^{l}(t\wedge\tau_{k,n})=&S_{t\wedge\tau_{k,n}}u_0-\i\int_0^{t\wedge\tau_{k,n}}S_{t\wedge\tau_{k,n}-s}\big(\lambda |u^{l}(s)|^{2\sigma}u^{l}(s)\big)\d s+I_{\tau_{k,n}}(\Phi(z,u^{l})-u^{l})(t\wedge\tau_{k,n})\\
     &+\int_0^{t\wedge\tau_{k,n}}\int_BS_{t\wedge\tau_{k,n}-s}\Big[\Phi(z,u^{l}_{s})-u^{l}_{s}+\i\sum_{j=1}^mz_jg_j(u^{l}_s)\Big]\nu(\d z)\d s, \quad \mathbb{P}\text{-a.s.}
     \end{split}
\end{align}
 Since
    $\triangle u^{l}(\tau_{k,n})= \int_B1_{[0,\tau_{k,n}]}S_{\tau_{k,n}-s}\Big[\Phi(z,u^{l}_{s\wedge\tau_{k,n}-})-u^{l}_{s\wedge\tau_{k,n}-}\Big]\tilde{N}(\{\tau_{k,n}\},\d z)$, for $l=k,n$
and $\Phi(z,u^{k}_{s\wedge\tau_{k,n}-})-u^{k}_{s\wedge\tau_{k,n}-}$ coincides with  $\Phi(z,u^{n}_{s\wedge\tau_{k,n}-})-u^{n}_{s\wedge\tau_{k,n}-}$ on $[0,\tau_{k,n}]$, we infer that
\begin{align}\label{eqn-X_n=X_m}
u^{k}=u^{n}\text{ on }[0,\tau_{k,n}].  \end{align}
Hence, by the contradiction argument, we can show that a.s.
\begin{align*}
\tau_k\leq \tau_n\ \text{if}\ k<n.
\end{align*}
 So the limit
$\lim_{n\to\infty}\tau_n=:\tau_\infty$ exists a.s. Let us denote
$\Omega_0=\{\omega:\lim_{n\to\infty}\tau_n=\tau_{\infty}\}$ and
note that $\mathbb{P}(\Omega_0)=1$.

 Now we define a local process
$(u(t))_{0\leq t<\tau_{\infty}}$ as follows. If
$\omega\notin\Omega_0$, set $u(t,\omega)=0$, for $0\leq t<\tau_{\infty}$.
If $\omega\in\Omega_0$, then for every $t<\tau_{\infty}(\omega)$, there exists  a number $n\in\mathbb{N}$
such that $t\leq \tau_n(\omega)$ and we
set $u(t,\omega)=u^n(t,\omega)$.  In view of \eqref{eqn-X_n=X_m} this process is well defined, $(u(t))_{t\in[0,\tau_{n}]}\in M^p_{\mathbb{F}}(Y_{\tau_{n}})$
and it satisfies for $t\in [0,T]$
       \begin{align}
     \begin{split}\label{prop-local-eq2}
     u(t\wedge\tau_{n})=&S_{t\wedge\tau_{n}}u_0-\i\int_0^{t\wedge\tau_{n}}S_{t\wedge\tau_{n}-s}\big(\lambda|u(s)|^{2\sigma}u(s)\big)   \d s+I_{\tau_{n}}(\Phi(z,u)-u)(t\wedge\tau_{n})\\
     &+\int_0^{t\wedge\tau_{n}}\int_BS_{t\wedge\tau_{n}-s}\Big[\Phi(z,u(s))-u(s)+\i\sum_{j=1}^mz_jg_j(u(s))\Big]\nu(\d z)\d s, \quad \mathbb{P}\text{-a.s.}
     \end{split}
\end{align}
where we used the fact that because of \eqref{eqn-X_n=X_m},  for all $0\leq t\leq T$,
\begin{align*}
                      I_{\tau_{n}}(\Phi(z,u)-u)(t)= I_{\tau_{n}}(\Phi(z,u^{n})-u^{n})(t).
\end{align*}
Furthermore, by the definition of the sequence
$\{\tau_n\}_{n=1}^{\infty}$ we infer that a.s. on the set $\{ \tau_{\infty}<\infty\}$,
\begin{align}\label{eq-334}
    \lim_{t \nearrow \tau_{\infty}}\|u\|_{Y_{t}}=\lim_{ n\rightarrow\infty} \|u\|_{Y_{\tau_{n}}}\geq \lim_{n}n=\infty\quad\mathbb{P}\text{-a.s.}
\end{align}
Using arguments similar to that in the proof of \cite[Proposition 1]{Brz+Zhu2016}, we can prove that $(u(t))_{t\in[0,\tau_{\infty})}$ is a maximal local mild solution of \eqref{NSS-integral-form}.

The uniqueness of the solution follows  from the construction of the solution and the uniqueness of the solution to the truncated equation.

 \end{proof}


  \section{Existence and uniqueness of a global solution to stochastic NLS in Marcus form}

  In this section, we shall prove the global existence of the original equation \eqref{SNLSE-Marcus}. To do that, first we show the $L^{2}(\RR^{n})$-norm of the solution is preserved. Then we establish uniform bounds for solutions of \eqref{SEE-trucated-1} in $L^{p}(0,T;L^r(\RR^n))$ by using again the deterministic and the stochastic Strichartz inequalities.
  \begin{proposition}\label{prop-L2-norm-preserved} Assume that Assumption \ref{assu-1} holds.
   Let $p\geq 2$, $0<\sigma<\frac{2}{n}$, $1\leq n<\infty$,
 $(p,r)$ be an admissible pair with $r=2\sigma+2$ and
          $u_0\in L^p(\Omega;L^2(\RR^n))$.  For $k\in \mathbb{N}$, let $u^{k}$ be the global mild solution of equation \eqref{SEE-trucated-1} with $R$ replaced by $k$. Then we have for all $t\in[0,T]$,
          \begin{align*}
          \|u^{k}(t)\|_{L^2(\RR^n)}=\|u_0\|_{L^2(\RR^n)}\quad \mathbb{P}\text{-a.s.}
          \end{align*}
  \end{proposition}
  \begin{proof}
    According to the Yosida approximating argument, we can always approximate \eqref{SEE-trucated-1} by equations having strong solutions. To do this, let us introduce
    \begin{align*}
         f^k_\mu(t)&=\mu(\mu I-A)^{-1}\theta_{k}(\|u^{k}\|_{Y_{t}})\lambda |u^{k}(t)|^{2\sigma}u^{k}(t),\\
         G_\mu^{k}(z,t)&=\mu(\mu I-A)^{-1}[\Phi(z,u^{k}(t-))-u^{k}(t-)],
         \end{align*}
         and
         \begin{align*}
         H_\mu^{k}(z,t)&=\mu(\mu I-A)^{-1}[\Phi(z,u^{k}(t-))-u^{k}(t-)+\i\sum_{j=1}^{m}z_jg_j(u^{k}(t))].
    \end{align*}
  Then the equation
   \begin{align*}
          &\d u^{k}_{\mu}(t)=i\Delta u_{\mu}^{k}(t)+ f^k_\mu(t)\d t+\int_{B} G_\mu^{k}(z,t)\tilde{N}(\d t,\d z)+\int_{B}  H_\mu^{k}(z,t)\nu(\d z)\d t,\\
          &u^{k}_{\mu}(0)=\mu(\mu I-A)^{-1}u(0),
 \end{align*}
 has a unique strong solution. By the properties of the Yosida approximation,  it's easy to see that
 \begin{align}\label{con-u-lambda}
      \lim_{\mu\rightarrow\infty}\E\sup_{t\in[0,T]}\|u^{k}_{\mu}(t)-u^{k}(t)\|_{L^2(\RR^n)}^{2}=0.
 \end{align}
 Define a function $\psi: L^2(\RR^n)\ni u\mapsto \frac12 \|u\|^2_{L^2(\RR^n)}=\frac12\int_{\RR^n} u(x)\overline{u(x)}dx\in\mathbb{R}$. Then we have
  \begin{align*}
  \psi'(u)(v)=\Re\langle u,v\rangle_{L^2(\RR^n)}=\int_{\RR^n}Re(u(x)\overline{v(x)})\d x.
  \end{align*}
  Applying the It\^{o} formula to the function $\psi$ and the strong solution $u^{k}_{\mu}$,
  we obtain
  \begin{align*}
  &\psi(u^{k}_{\mu}(t))-\psi(u_{\mu}^{k}(0))\\
  =&\int_0^{t}\langle \psi'(u^{k}_{\mu}(s)),\i(\Delta u^{k}_{\mu}(s)-f^{k}_{\mu}(s)\rangle_{L^2} \d s\\
    &+\int_0^{t}\int_B\Big[\psi(u_{\mu}^{k}(s-)+G_{\mu}^{k}(z,t))-\psi(u_{\mu}^{k}(s-))\Big]\tilde{N}(\d s,\d z)\\
    &+\int_0^{t}\int_B\Big[\psi(u_{\mu}^{k}(s-)+G_{\mu}^{k}(z,t))-\psi(u_{\mu}^{k}(s-))-\langle \psi'(u^{k}_{\mu}(s)),G_\mu^{k}(z,t)\rangle _{L^2(\RR^n)}
\Big]\nu(\d z)\d s\\
  &+\int_0^{t}\int_B\Big[\langle \psi'(u_{\mu}^{k}(s)),H_{\mu}^{k}(z,t)\rangle _{L^2(\RR^n)}
\Big]\nu(\d z)\d s\\
  =&\int_0^{t}\Re\langle u_{\mu}^{k}(s),\i\Delta u^{k}_{\mu}(s)\rangle_{L^2(\RR^n)}\d s
        -\int_0^{t} \Re\langle u_{\mu}^{k}(s),\i f^{k}_{\mu}(s) \rangle_{L^2(\RR^n)} \d s\\
            &+\int_0^{t}\int_B \Big[\frac12\Big\|u_{\mu}^{k}(s-)+G_{\mu}^{k}(z,t)\Big\|^2_{L^2(\RR^n)}-\frac12\|u_{\mu}^{k}(s-)\|^2_{L^2(\RR^n)}\Big]\tilde{N}(\d s,\d z)\\
        &+\int_0^{t}\int_B  \Big[\frac12\Big\|u_{\mu}^{k}(s-)+G_{\mu}^{k}(z,t)\Big\|^2_{L^2(\RR^n)}-\frac12\|u_{\mu}^{k}(s-)\|^2_{L^2(\RR^n)}-\Re\big{\langle}u^{k}_{\mu}(s), G_{\mu}^{k}(z,t)\big{\rangle}_{L^2(\RR^n)}\big]\nu(\d z)\d s\\
        &+\int_0^{t}\int_B  \Re\big{\langle} u^{k}_{\mu}(s),H_{\mu}^{k}(z,s)\big{\rangle}_{L^2(\RR^n)}  \nu(\d z)\d s\\
         =&
        -\int_0^{t} \Re\langle u_{\mu}^{k}(s),\i f^{k}_{\mu}(s) \rangle_{L^2(\RR^n)} \d s\\
            &+\int_0^{t}\int_B \Big[\frac12\Big\|u_{\mu}^{k}(s-)+G_{\mu}^{k}(z,t)\Big\|^2_{L^2(\RR^n)}-\frac12\|u_{\mu}^{k}(s-)\|^2_{L^2(\RR^n)}\Big]\tilde{N}(\d s,\d z)\\
        &+\int_0^{t}\int_B  \Big[\frac12\Big\|u_{\mu}^{k}(s-)+G_{\mu}^{k}(z,t)\Big\|^2_{L^2(\RR^n)}-\frac12\|u_{\mu}^{k}(s-)\|^2_{L^2(\RR^n)}\\
       &\hspace{5cm} +\Re\big{\langle}u^{k}_{\mu}(s), \i\sum_{j=1}^{m}z_{j}g_{j}(u^{k}(s))\big{\rangle}_{L^2(\RR^n)}\big]\nu(\d z)\d s,
  \end{align*}
  where we used the fact that $\Re\langle u(s),i\triangle u(s)\rangle_{L^2}=0$, since $i\Delta$ is skew-self-adjoint in $L^2(\RR^n)$.

By using \eqref{con-u-lambda}, the It\^{o} isometry property  of the stochastic integral and the Lebesgue Dominated Convergence Theorem, we get as $\mu\rightarrow\infty$ (passing to a subsequence if necessary)
\begin{align*}
     & \int_0^{t} \Re\langle u_{\mu}^{k}(s),\i \theta_{k}(\|u^{k}\|_{Y_{s}})f^{k}_{\mu}(s) \rangle_{L^2(\RR^n)} \d s\\
     &\longrightarrow \int_0^{t} \Re\langle u^{k}(s),\i \theta_{k}(\|u^{k}\|_{Y_{s}})\lambda |u^{k}(s)|^{2\sigma}u^{k}(s) \rangle_{L^2(\RR^n)} \d s\quad\mathbb{P}\text{-a.s.}\\
    &  \int_0^{t}\int_B \Big[\frac12\Big\|u_{\mu}^{k}(s-)+G_{\mu}^{k}(z,t)\Big\|^2_{L^2(\RR^n)}-\frac12\|u_{\mu}^{k}(s-)\|^2_{L^2(\RR^n)}\Big]\tilde{N}(\d s,\d z)\\
    &\longrightarrow \int_0^{t}\int_B \Big[\frac12\|\Phi(z,u^{k}(s-)) \|^2_{L^2(\RR^n)}-\frac12\|u^{k}(s-)\|^2_{L^2(\RR^n)}\Big]\tilde{N}(\d s,\d z)\quad\mathbb{P}\text{-a.s.}\\
    &\int_0^{t}\int_B  \Big[\frac12\Big\|u_{\mu}^{k}(s-)+G_{\mu}^{k}(z,t)\Big\|^2_{L^2(\RR^n)}-\frac12\|u_{\mu}^{k}(s-)\|^2_{L^2(\RR^n)}\\
       &\hspace{4cm} +\Re\big{\langle}u^{k}_{\mu}(s), \theta_{k}(\|u^{k}\|_{Y_{t}})\i\sum_{j=1}^{m}z_{j}g_{j}(u^{k}(s))\big{\rangle}_{L^2(\RR^n)}\big]\nu(\d z)\d s\\
        &\longrightarrow \int_0^{t}\int_B  \Big[\frac12\|\Phi(z,u^{k}(s-)) \|^2_{L^2(\RR^n)}-\frac12\|u^{k}(s-)\|^2_{L^2(\RR^n)}
       \\
       &\hspace{4cm} +\Re\big{\langle}u^{k}(s),\theta_{k}(\|u^{k}\|_{Y_{t}}) \i\sum_{j=1}^{m}z_{j}g_{j}(u^{k}(s))\big{\rangle}_{L^2(\RR^n)}\big]\nu(\d z)\d s\quad\mathbb{P}\text{-a.s.}
\end{align*}
Therefore, using these limiting results and \eqref{con-u-lambda}, we obtain
 \begin{align*}
  &\psi(u^{k}(t))-\psi(u^{k}(0))
   =- \int_0^{t} \Re\langle u^{k}(s),\i \theta_{k}(\|u^{k}\|_{Y_{s}})\lambda |u^{k}(s)|^{2\sigma}u^{k}(s) \rangle_{L^2(\RR^n)} \d s\\
            &+\int_0^{t}\int_B \Big[\frac12\|\Phi(z,u^{k}(s-)) \|^2_{L^2(\RR^n)}-\frac12\|u^{k}(s-)\|^2_{L^2(\RR^n)}\Big]\tilde{N}(\d s,\d z)\\
        &+\int_0^{t}\int_B  \Big[\frac12\|\Phi(z,u^{k}(s-)) \|^2_{L^2(\RR^n)}-\frac12\|u^{k}(s-)\|^2_{L^2(\RR^n)}\\
       &\hspace{4cm} +\Re\big{\langle}u^{k}(s), \i\sum_{j=1}^{m}z_{j}g_{j}(u^{k}(s))\big{\rangle}_{L^2(\RR^n)}\big]\nu(\d z)\d s.
  \end{align*}
 Let us observe that by \eqref{eqn-norm conservation}
  \begin{align}
            \|\Phi(z,u)\|^2_{L^2(\RR^n)}=\|u\|^2_{L^2(\RR^n)},\;\;\text{for all
} u \in L^2(\RR^n), \; z\in \mathbb{R}^m.
  \end{align}
  Since   $\Re\langle i u,u\rangle=0$ for $u \in \mathbb{C}$, we have
  \begin{align*}
  \Re\big{\langle} u^{k}(s),\i\sum_{j=1}^m z_jg_j(u^{k}(t))\big{\rangle}_{L^2(\RR^n)}=0.
  \end{align*}
  and
    \begin{align*}
  \Re\big{\langle} u^{k}(s),\i \theta_{k}(\|u^{k}\|_{Y_{s}}) \lambda |u^{k}(s)|^{2\sigma}u^{k}(s)\big{\rangle}_{L^2(\RR^n)}=0.
  \end{align*}
  Consequently, we obtain
  \begin{align*}
             \psi(u^{k}(t))-\psi(u_0)
        =0,
  \end{align*}
  which finishes the proof.

  \end{proof}


  We are finally ready to present the proof of Theorem \ref{main-theo}. Similar arguments of Schr\"odinger equations with gaussian noise are given in \cite{Bou+Deb1999} and \cite{Hornung2018}. Our approach is different from these papers as we introduce some stopping times which we think are essential in the proof.

  \begin{proof}[Proof of Theorem \ref{main-theo}]
  Let us choose and fix  $p\geq 2$, $0<\sigma<\frac{2}{n}$, $r=2\sigma+2$ such that    $(p,r)$ be an admissible pair, i.e. the conditions (\ref{eqn-admissible pair-01}-\ref{eqn-admissible pair-02}) are satisfied.    And finally    let us choose and fix  an initial data $u_0$ such that $u_0\in L^p(\Omega,\mathcal{F}_{0};L^2(\RR^n))$. Let us recall that $\lambda$ is the coefficient from equation \eqref{SNLSE-Marcus}.
    Let us  choose and fix $T>0$. It is sufficient to show that the stopping time $\tau_\infty$ from Proposition \ref{prop-local-mind-main} is equal, $\mathbb{P}$-almost surely,  to $T$.

  Let $(u^{k})_{k\in\mathbb{N}}$ be the sequence of solutions of \eqref{SEE-trucated-1} provided by Proposition \ref{prop-truncated equation} with $R$ replaced by $k$ as in Proposition \ref{prop-local-mind-main}.

  \emph{Step 1.} We first prove that $u^{k}$ is uniformly bounded in $L^p(\Omega,L^p(0,T;L^r(\mathbb{R}^n)))$, i.e. we will show that there exists a  constant $C>0$ such that
   \begin{align}\label{uni-bound-of-local-solution}
     \sup_k \E \|u^{k}\|^p_{L^{p}(0,T;L^r(\RR^n))}\leq C.
  \end{align}
Recall that by \eqref{eqn-Psi}  $u^k$ is given by
 \begin{align*}
     u^{k}(t)=S_tu_0+\Psi_1^{k}(u^{k})(t)+\Psi_2^{k}(u^{k})(t)+\Psi_3^{k}(u^{k})(t),\;\; t\in[0,T].
  \end{align*}
By Proposition \ref{prop-str-ineq-11}, Lemma \ref{lem-stoch-lip} and $L^{2}(\RR^{n})$-norm-preserving property in Proposition \ref{prop-L2-norm-preserved}, we have, for $q\geq p$,
  \begin{align}
  \begin{split}\label{th-proof-est-psi-2}
  \E \|\Psi_2^{k}(u_k)(\cdot)\|^q_{L^{p}(0,T;L^r(\RR^n))}
    &\leq C_{q}\, \EE\Big(\int_0^T \int_B \|\Phi(z,u^k(s))-u^k(s)\|_{L^2(\RR^{n})}^2\nu(\d z)\d s\Big)^{\frac{q}{2}}\\
  &\quad+C_{q}\,\EE\Big(\int_0^T \int_B\|\Phi(z,u^k(s))-u^k(s)\|_{L^2(\RR^{n})}^q\nu(\d z)\d s\Big)\\
   & \leq C_{q,m}(T+T^{\frac{q}{2}})\E\|u^{k}\|^q_{L^{\infty}([0,T];L^2(\RR^n))}
   =C_{q,m}(T+T^{\frac{q}{2}})\E\|u_0\|_{L^2(\RR^n)}^q<\infty.
  \end{split}
  \end{align}
   Let us fix $\omega\in \Omega$ and take $T_k(\omega)\in (0,T]$ whose value will be determined later on. Applying similar arguments as in Proposition \ref{prop_sto_stri_1}, \ref{prop_sto_stri_2} and \ref{prop_sto_stri_3}, we have
  \begin{align}
  &\hspace{3truecm} \|u^{k}\|^p_{L^{p}(0,T_k;L^r(\RR^n))}\nonumber\\
  &\leq \|S_tu_0\|^p_{L^{p}(0,T_k;L^r(\RR^n))}+\|\Psi_1^{k}(u^{k})\|^p_{L^{p}(0,T_k;L^r(\RR^n))}+\|\Psi_2^{k}(u^{k})\|^p_{L^{p}(0,T_k;L^r(\RR^n))}+\|\Psi_3^{k}(u^{k})\|^p_{L^{p}(0,T_k;L^r(\RR^n))}\nonumber\\
          & \leq C\|u_0\|^p_{L^2(\RR^n)}+ C|\lambda|T_k^{(1-\frac{n\sigma}{2})p}\|u^{k}\|^{p(2\sigma+1)}_{L^{p}(0,T_k;L^r(\RR^n))}+\|\Psi_2^{k}(u^{k})\|^p_{L^{p}(0,T_k;L^r(\RR^n))}
         +C_{m}T_k^p\|u_k\|^p_{L^{\infty}([0,T_k];L^2(\RR^n))}\nonumber\\
         &=C\|u_0\|^p_{L^2(\RR^n)}+ C|\lambda|T_k^{(1-\frac{n\sigma}{2})p}\|u^{k}\|^{p(2\sigma+1)}_{L^{p}(0,T_k;L^r(\RR^n))}+\|\Psi_2^{k}(u^{k})\|^p_{L^{p}(0,T_k;L^r(\RR^n))}
         +C_{m}T_k^p\|u_0\|^p_{L^2(\RR^n)}\nonumber\\
         &=\Big[\big(C+C_{m}T_k^p\big)\|u_0\|^p_{L^2(\RR^n)}+\|\Psi_2^{k}(u^{k})\|^p_{L^{p}(0,T_k;L^r(\RR^n))}\Big]+C|\lambda|T_k^{(1-\frac{n\sigma}{2})p}\|u^{k}\|^{p(2\sigma+1)}_{L^{p}(0,T_k;L^r(\RR^n))}\label{TH-ineq-equation}\\
         & \leq M_{k}^{T_k}(\omega)+C|\lambda|T_k^{(1-\frac{n\sigma}{2})p}\|u^{k}\|^{p(2\sigma+1)}_{L^{p}(0,T_k;L^r(\RR^n))},\nonumber
  \end{align}
  where
  \begin{align}\label{eqn-M_k^T_k}
  M_{k}^{T_k}(\omega):=\big(C+C_{m}T_k^p\big)\|u_0\|^p_{L^2(\RR^n)}+\|\Psi_2^{k}(u^{k})\|^p_{L^{p}(0,T_k;L^r(\RR^n))}.
  \end{align}
 Also let us denote
  \begin{align}\label{eqn-M_k(omega)}
  M_{k}(\omega):=\big(C+C_{m}T^p\big)\|u_0\|^p_{L^2(\RR^n)}+\|\Psi_2^{k}(u^{k})\|^p_{L^{p}(0,T;L^r(\RR^n))}.
  \end{align}
  Consider  a function $f$ defined by  $$f(x)=M+\frac{1}{4(2M)^{2\sigma}}x^{2\sigma+1}-x,\quad x\geq 0.$$
 Here $M>0$. Obviously  $f^{\prime\prime}(x)\geq 0$, for all $x\in [0,\infty)$, it
is convex on $[0,\infty)$. Notice also that
   \begin{align*}
   &   f(0)=M>0,\\
   &   f(2M)=M+\frac{1}{4(2M)^{2\sigma}}(2M)^{2\sigma+1}-2M=-\frac{M}{2}<0,\\
   &f\Big( 4^{\frac1{2\sigma} }\cdot 2M\Big)=M+\frac{1}{4(2M)^{2\sigma}}\Big(4^{\frac1{2\sigma} }\cdot 2a\Big)^{2\sigma+1}-4^{\frac1{2\sigma} }\cdot 2M=M>0.
   \end{align*}
By the intermediate value theorem, there exists two points $c_1,c_2$ with $0<c_1<2M<c_2<4^{\frac1{2\sigma} }\cdot 2M<\infty$ at which $f(c_1)=f(c_2)=0$. Therefore, we can conclude that
\begin{align}\label{TH-proof-property-f}
f(x)\geq 0 \text{ if and only if } 0\leq x\leq c_1\text{ or }x\geq c_2.
\end{align}
Recalling  that by assumptions $1-\frac{n\sigma}{2}>0$,  we choose  the required value $T_k(\omega)$, $\omega \in \Omega$,  to be
\begin{equation}\label{eqn-T_k(omega)}
  T_k(\omega)= T \wedge (4C|\lambda|(2M_k(\omega))^{2\sigma})^{-\frac1{p(1-\frac{n\sigma}{2})}}.
  \end{equation}
With this choice we infer that  $C|\lambda|T_k(\omega)^{(1-\frac{n\sigma}{2})p}\leq \frac{1}{4(2M_k(\omega))^{2\sigma}}$.
Note that the above definition of $T_k$ is equivalent to $T_k\leq T$ and
\[
C|\lambda|T_k^{(1-\frac{n\sigma}{2})p} (M_k)^{2\sigma} \leq \frac{1}{4 \times 2^{2\sigma}}.
\]
Replacing $M_k$ by $M_k^{T_k}$ and using \eqref{eqn-M_k^T_k}, the above means
\[
C|\lambda|T_k^{(1-\frac{n\sigma}{2})p} \Bigl[ \big(C+C_{m}T_k^p\big)\|u_0\|^p_{L^2(\RR^n)}+\|\Psi_2^{k}(u^{k})\|^p_{L^{p}(0,T_k;L^r(\RR^n))} \Bigr]^{2\sigma} \leq \frac{1}{4 \times 2^{2\sigma}}.
\]
This inequality suggests that we can replace $T_k$ which is not a stopping time by a stopping time $\sigma_k$ defined by
\begin{align}\label{eqn-sigma_k}
\sigma_k:= \inf\Bigl\{t\in [0,T]: C|\lambda|t^{(1-\frac{n\sigma}{2})p} \Bigl[ \big(C+C_{m}t^p\big)\|u_0\|^p_{L^2(\RR^n)}+\|\Psi_2^{k}(u^{k})\|^p_{L^{p}(0,t;L^r(\RR^n))} \Bigr]^{2\sigma} >  \frac{1}{4 \times 2^{2\sigma}} \Bigr\}.
\end{align}
Observe that $T_k\leq \sigma_k$. Clearly, \eqref{TH-ineq-equation} still holds with $T_k$ replaced by $\sigma_k$. That is
\begin{align*}
        \|u^{k}\|^p_{L^{p}(0,\sigma_k;L^r(\RR^n))}&\leq M_{k}^{\sigma_k}(\omega)+C|\lambda|\sigma_k^{(1-\frac{n\sigma}{2})p}\|u^{k}\|^{p(2\sigma+1)}_{L^{p}(0,\sigma_k;L^r(\RR^n))}\\
        &\leq M_{k}^{\sigma_k}(\omega)+ \frac{1}{4(2M_k^{\sigma_k})^{2\sigma}}\|u^{k}\|^{p(2\sigma+1)}_{L^{p}(0,\sigma_k;L^r(\RR^n))}.
\end{align*}
 Since $X_t:=\|u^{k}\|^p_{L^{p}(0,t;L^r(\RR^n))}$ is continuous in $t\in[0,\infty)$ and $X_0=0$,  by the property \eqref{TH-proof-property-f} of the function $f$ discussed above, we infer that
  \begin{align}\label{eqn-2M_k}
      &\|u^{k}\|^p_{L^{p}(0,\sigma_k;L^r(\RR^n))}\leq c_1 \leq 2M_k^{\sigma_k}(\omega).
  \end{align}
Substituting $M_k^{\sigma_k}$ into the above  gives
  \begin{align*}
      \|u_k\|^p_{L^{p}(0,\sigma_k;L^r(\RR^n))}\leq 2\big(C+C_{m}(\sigma_k)^p\big)\|u_0\|^p_{L^2(\RR^n)}+2\|\Psi_2^{k}(u^{k})(\cdot)\|^p_{L^{p}(0,\sigma_k;L^r(\RR^n))}.
  \end{align*}
Then we define a sequence $\sigma_k^j$ for $j=1, \cdots$ as follows. For $j=1$ we put $\sigma_k^1=\sigma_k$. For $j=2$ we define  $\sigma_k^2$ is the infimum of times after $\sigma_k^1$ such that the above condition holds but the initial time $0$ is replaced by the initial time $\sigma_k^1$. To be more precise, for $j=1,2,\cdots$, we define
\begin{align*}
\sigma_k^{j+1}:= \inf\Bigl\{&t\in[\sigma_k^j,T]: \\
&C|\lambda|(t-\sigma_k^j)^{(1-\frac{n\sigma}{2})p} \Bigl[ \big(C+C_{m}(t-\sigma_k^j)^p\big)\|u_k(\sigma_k^j)\|^p_{L^2(\RR^n)}+\|\Psi_2^{k}(u^{k})\|^p_{L^{p}(\sigma_k^j,t;L^r(\RR^n))} \Bigr]^{2\sigma} > \frac{1}{4 \times 2^{2\sigma}} \Bigr\}.
\end{align*}
By the definition of  $\sigma_k^j$, we infer $\sigma_k^{j+1}-\sigma_k^j \geq T_k$, for each $j$. Let $N=[\frac{T}{T_k}]$. Then we can see that $\sigma_k^{j}=T$, for $j=N+1,N+2,\cdots$.
Since $\sigma_k^j$ is a stopping time, the following equality holds   for $t\geq \sigma_k^j$,
\begin{align}
  \begin{split}\label{SEE-trucated-end}
     u^{k}(t)=&S_{t-\sigma_k^j}(u^k(\sigma_k^j))-\i\int_{\sigma_k^j}^tS_{t-s}(\theta_k(\|u^{k}\|_{Y_s})\big(\lambda |u^{k}(s)|^{2\sigma}u^{k}(s) \big)\d s\\
     &+\int_{\sigma_k^j}^t\int_BS_{t-s}\Big[\Phi(z,u^{k}(s-))-u^{k}(s-)\Big]\tilde{N}(\d s,\d z)\\
     &+\int_{\sigma_k^j}^t\int_BS_{t-s}\Big[\Phi(z,u^{k}(s))-u^{k}(s)+\i\sum_{j=1}^mz_jg_j(u(s))\Big]\nu(\d z)\d s.
     \end{split}
\end{align}
Hence,  applying similar arguments as above inductively gives,   for $j=1,\cdots,N$,
  \begin{align*}
    \|u_k\|^p_{L^{p}(\sigma_k^j,\sigma_k^{j+1};L^r(\RR^n))}&\leq 2\big(C+C_{m}(\sigma_k^{j+1}-\sigma_k^j)^p\big)\|u_k(\sigma_k^j)\|^p_{L^2(\RR^n)}+2\|\Psi_2^{k}(u^{k})(\cdot)\|^p_{L^{p}(\sigma_k^j,\sigma_k^{j+1};L^r(\RR^n))}\\
    &\leq 2 M_k(\omega),
  \end{align*}
  where we used the $L^{2}(\RR^{n})$-norm-preserving property of $u_k$ from Proposition \ref{prop-L2-norm-preserved}. Thus,
  \begin{align*}
   \|u_k\|^p_{L^{p}(0,T;L^r(\RR^n))}&\leq   \|u_k\|^p_{L^{p}(0,\sigma_k^{1};L^r(\RR^n))} +\sum_{j=1}^{N}  \|u_k\|^p_{L^{p}(\sigma_k^j,\sigma_k^{j+1};L^r(\RR^n))}
   \leq 2(\frac{T}{T_k}+1)M_k\\
   &= 2M_k+2T (4C|\lambda|(2M_k)^{2\sigma})^{\frac1{(1-\frac{n\sigma}{2})p}} M_k\\
   &=2M_k+C_{n,p,\lambda,\sigma,T} (M_k)^{\frac{4\sigma}{(2-n\sigma)p}+1}.
   \end{align*}
Put $\theta=\frac{4\sigma}{(2-n\sigma)p}+1$. By applying \eqref{th-proof-est-psi-2} with $q=p, p\theta$ respectively, we obtain
    \begin{align}
    \begin{split}\label{eq-unif-est}
      \E\|u^{k}\|^p_{L^{p}(0,T;L^r(\RR^n))}&\leq 2\E M_k+C_{n,p,\lambda,\sigma,T} \E (M_k)^{\theta}\\
      & \leq 2\big(C+C_{m}T^p\big)\E\|u_0\|^p_{L^2(\RR^n)}+2\E \|\Psi_2^{k}(u^{k})(\cdot)\|^p_{L^{p}(0,T;L^r(\RR^n))}\\
      &\quad+C_{n,p,\lambda,\sigma,T} \big(C+C_{m}T^p\big)^{\theta} \E\|u_0\|_{L^2(\R^n)}^{p\theta}+C_{n,p,\lambda,\sigma,T} \E \|\Psi_2^{k}(u^{k})(\cdot)\|^{p\theta}_{L^{p}(0,T;L^r(\RR^n))}\\
      &\leq
      C_{m,p,T}\E\|u_0\|_{L^2(\RR^n)}^p+C_{m,n,p,\lambda,\sigma,T}\E\|u_0\|_{L^2(\RR^n)}^{p\theta}
      \\
      & \leq C_{m,n,p,\lambda,\sigma,T}\Big(\E\|u_0\|_{L^2(\RR^n)}^p+\E\|u_0\|_{L^2(\RR^n)}^{\frac{4\sigma }{2-n\sigma}+p}\Big) :=\rho.
      \end{split}
  \end{align}
  Here the constant $\rho$ depends on $T$ and $\|u_0\|_{L^2
  (\R^n)}$ but is independent of $k$. Hence we proved \eqref{uni-bound-of-local-solution}.

 \emph{ Step 2. } Recall from Proposition \ref{prop-local-mind-main} that  $\tau_{k}$ is defined by   $
        \tau_k=\inf\{t\in[0,T]:\|u^{k}\|_{Y_{t}}> k\}$.
  By using the result we obtained in the first step, we deduce that
  \begin{align*}
  \mathbb{P}(\tau_k=T)&= \mathbb{P}\{ \sup_{0\leq t\leq T}\|u^{k}(t)\|_{L^2(\RR^n)} + \|u^{k}\|_{L^p(0,T;L^r(\RR^n))} \leq k\}\\
  &\geq 1-\frac{2^{p}\E\sup_{0\leq t\leq T} \|u^{k}(t)\|^{p}_{L^2(\RR^n)}+2^{p}\E  \|u^{k}\|^{p}_{L^p(0,T;L^r(\RR^n))}}{k^{p} }\\
  &\geq 1- \frac{2^{p}\E \|u_{0}\|^{p}_{L^{2}(\RR^{n})}+2^{p}\rho}{k^{p}}.
  \end{align*}
Hence we have
\begin{align*}
 \mathbb{P}(\tau_{\infty}=T)\geq \mathbb{P}\Big(\cup_{k}\{\tau_k=T\}\Big)=\lim_{k\rightarrow\infty}\mathbb{P}(\tau_k=T)=1.
\end{align*}
Thus we infer that $\tau_\infty \geq T$,  $\mathbb{P}$-almost surely. Since $T$ was arbitrary,
this shows $u(t)$, $t\in[0,\infty)$ is a global mild solution.
Moreover, by Proposition \ref{prop-truncated equation}, \ref{prop-local-mind-main} and \ref{prop-L2-norm-preserved}, $\mathbb{P}$-a.s.  the following holds:
 $\|u(t)\|_{L^2(\RR^n)}=\|u_0\|_{L^2(\RR^n)}$ for all $t\in [0,\infty)$,   $u\in D([0,\infty);L^2(\RR^n))$
and, for every $T>0$, $u\in L^p\big(\Omega;L^{\infty}(0,T;L^2(\RR^n))\cap L^p(0,T;L^r(\RR^n))\big)$.

  \end{proof}


  \begin{acknowledgements} A very preliminary version of this paper, containing Proposition \ref{prop-str-ineq-11}, about the stochastic Strichartz estimates) was presented by the first named author at the workshop ``Stochastic Analysis, L\'evy processes and (B)SDEs'' held in Innsbruck in October 2011 and organized by C. Geis, S. Geis and E. Hausenblas. He would like to thank the organizers for their kind invitation.
  \end{acknowledgements}

\end{document}